\documentclass[11pt]{jdg-p}

\usepackage[T1]{fontenc}
\usepackage[latin1]{inputenc}

\usepackage{setspace}
\usepackage{indentfirst}

\usepackage{color}
\definecolor{red}{rgb}{1,0,0}
\definecolor{green}{rgb}{0,1,0}
\definecolor{blue}{rgb}{0,0,1}
\definecolor{refkey}{gray}{.625}
\definecolor{labelkey}{gray}{.625}

\let\oldmarginpar\marginpar
\renewcommand\marginpar[1]{\-\oldmarginpar[\raggedleft\footnotesize #1]%
{\raggedright\footnotesize #1}}

\usepackage{graphicx}
\usepackage[all]{xy}
\def\dar[#1]{\ar@<2pt>[#1]\ar@<-2pt>[#1]}

\usepackage{amsthm}

\theoremstyle{plain}
\newtheorem{theorem}{Theorem}[section]
\newtheorem{lemma}[theorem]{Lemma}
\newtheorem{corollary}[theorem]{Corollary}
\newtheorem{definition}[theorem]{Definition}
\newtheorem{proposition}[theorem]{Proposition}
\newtheorem{main}{Theorem}

\newtheorem*{thm*}{Theorem}

\newtheoremstyle{slanted}
  {3pt}
  {3pt}
  {\slshape}
  {}
  {\bfseries}
  {.}
  {.5em}
  {}

\theoremstyle{slanted}
\newtheorem{example}[theorem]{Example}
\newtheorem{remark}[theorem]{Remark}

\theoremstyle{definition}

\newtheorem*{note*}{Note}

\theoremstyle{remark}

\DeclareMathOperator{\id}{id}

\DeclareMathOperator{\Hom}{Hom}

 \DeclareMathOperator{\pr}{pr}

\newcommand{\CC}{\mathbb{C}}
\newcommand{\reals}{\mathbb{R}}

\newcommand{\ZZ}{\mathbb{Z}}

\newcommand{\thalf}{\tfrac{1}{2}}

\newcommand{\smalcirc}{\mbox{\tiny{$\circ$}}}

\newcommand{\cinf}[1]{C^{\infty}(#1)}
\newcommand{\sections}[1]{ {\pmb{ \Gamma}} (#1)}

\newcommand{\toto}{\rightrightarrows}

\newcommand{\cala}{\mathcal{A}}

\newcommand{\ba}[2]{[#1,#2]}
\newcommand{\bas}[2]{[#1,#2]_*}

\newcommand{\ip}[2]{\langle #1 , #2 \rangle}

\newcommand{\abs}[1]{\left\vert#1\right\vert}

\newcommand{\As}{A^*}

\newcommand{\pairing}[2]{\left\langle #1  |  #2 \right\rangle}
\newcommand{\CIM}{C^{\infty}(M)}
\newcommand{\inserts}{\iota}
\newcommand{\contract}{\lrcorner\,}

\newcommand{\set}[1]{\left\{#1\right\}}

\newcommand{\Poissonbracket}[1]{\left \{ #1\right \}}

\newcommand{\thetagroup}{\Theta}
\newcommand{\thetaalgebra}{\mathfrak{\theta}}

\newcommand{\thetaalgebrastar}{\mathfrak{\theta}^*}
\newcommand{\galgebra}{\mathfrak{g}}
\newcommand{\galgebrastar}{\mathfrak{g}^*}
\newcommand{\ggroup}{G}
\newcommand{\inverse}[1]{{#1}^{-1}}
\newcommand{\crossedmoduletwogroup}{ {\ggroup\ltimes\thetagroup} }

\newcommand{\crossedmodulealgebra}{\galgebra\ltimes  \thetaalgebra}

\newcommand{\gpmulti}{\diamond}
\newcommand{\gpoidmulti}{\star}
\newcommand{\gpinverse}[1]{{#1}_{\diamond}^{-1}}
\newcommand{\gpoidinverse}[1]{{#1}^{-1}_{\star}}
\newcommand{\unit}{{\bf{1}}}
\newcommand{\source}{{\bf{s}}}
\newcommand{\target}{{\bf{t}}}

\newcommand{\Adjoint}[1]{\mathrm{Ad}_{#1}}
\newcommand{\adjoint}[1]{\mathrm{ad}_{#1}}

\newcommand{\coAdjoint}[1]{\mathrm{Ad}^*_{#1}}

\newcommand{\phiUprotate }{\phi^{\scriptscriptstyle T}}
\newcommand{\Id}{{\bf{1}}}

\newcommand{\ddelta}{    { \delta}}
\newcommand{\deta}{    { \eta}}
\newcommand{\dlambda}{    { \lambda}}
\newcommand{\skewdelta}{{\delta}}
\newcommand{\domega}{  { \omega}}

\newcommand{\Rmatrix}{\mathbf{r}}

\newcommand{\gltwo}{\mathfrak{gl}(2)}

\newcommand{\gthetaPi}{\mathbf{\Pi}}
\newcommand{\gthetaV}{\mathbf{V}}
\newcommand{\gthetaC}{\mathbf{C}}

\newcommand{\intsigma}{\widehat{\sigma}}
\newcommand{\dsigma}{\sigma}
\newcommand{\inteta}{\widehat{\eta}}
\newcommand{\intlambda}{\widehat{\lambda}}
\newcommand{\crossedmoduletriple}[3]{(#1\stackrel{#2}{\rightarrow}#3)}
\newcommand{\Lietwo}[2]{ {#2 \ltimes #1}}

\newcommand{\groupcrossedmoduletriple}[3]{(#1\stackrel{#2}{  \rightarrow }#3)}
\newcommand{\sltwo}{\mathfrak{sl}(2)}

\newcommand{\Phistar}{\Phi_*}

\newcommand{\groupaction}{ {  \triangleright }}

\newcommand{\moduleaction}{ \triangleright}
\newcommand{\polyvectorfields}[2]{\mathfrak{X}^{#1}(#2)}
\newcommand{\strictmultiplicative}[2]{\mathfrak{X}^{#1}_{\mathrm{mult}}(#2)}

\newcommand{\phipush}{D_{\phi}}
\newcommand{\rhopush}{D_{\rho}}

\newcommand{\minuspower}[1]{(-1)^{#1}}

\newcommand{\gradedcommutator}[2]{\lfloor #1,#2 \rfloor}

\newcommand{\gpleftmove}[1]{\overleftarrow{#1}^{\mathrm{\scriptscriptstyle gp}}}
\newcommand{\gpoidleftmove}[1]{\overleftarrow{#1}}
\newcommand{\gpoidrightmove}[1]{\overrightarrow{#1}}
\newcommand{\LiealgebroidoverG}{A}
\newcommand{\thetaalgebradegone}{\mathrm{\theta}}
\newcommand{\thetaalgebrastardegminustwo}{\mathrm{\theta}^*}
\newcommand{\galgebradegzero}{\mathfrak{g}}
\newcommand{\galgebrastardegminusone}{\mathfrak{g}^*}
\newcommand{\phimap}{\phi}
\newcommand{\bracketmapbracket}[2]{[#1,#2]}
\newcommand{\actionmapaction}[2]{#1 \moduleaction #2}
\newcommand{\homotopymap}{h}

\newcommand{\cohomotopymap}{\tilde{\eta}}
\newcommand{\symmetricproduct}{{\scriptstyle \odot}\,}

\newcommand{\symmetricalgebra}{\mathcal{S}}
\newcommand{\symmetricalgebraa}{{\mathcal{S}}^\bullet}

\newcommand{\Sbullet}{S^\bullet}

\newcommand{\tobefilledin}{\,\stackrel{\centerdot}{}\,}
\newcommand{\Linf}{L_{\infty}}
\newcommand{\degreesubspace}[2]{{#1}^{(#2)}}
\newcommand{\Vs}{V^*}

\newcommand{\shiftby}[2]{{#1}{\scriptstyle{[#2]}}}
\newcommand{\fullaction}[1]{ {#1}} 

\newcommand{\ipG}[2]{(#1,#2)}
\newcommand{\plankconstant}{\hbar}

\newcommand{\groupdifferential}[1]{\partial^{\mathrm{\scriptscriptstyle gp}}_{#1}}
\newcommand{\groupoiddifferential}[1]{\partial^{\mathrm{\scriptscriptstyle gpd}}_{#1}}
\newcommand{\wedgedelta}{\widehat{\delta}}
\newcommand{\wedgeomega}{\widehat{\omega}}

\newcommand{\derivesh}{\check{h}}
\newcommand{\derivesaction}{\check{a}}
\newcommand{\derivesbracket}{\check{b}}
\newcommand{\derivesvarphi}{\check{\phi}}
\newcommand{\derivescoh}{\check{\eta}}
\newcommand{\derivescoaction}{\check{\alpha}}
\newcommand{\derivescobracket}{\check{\epsilon}}
\newcommand{\derivescovarphi}{\check{\phi}}

\begin{document}

\title{Poisson 2-Groups}

\author{Zhuo Chen}
\address{Department of Mathematics, Tsinghua University}
\email{zchen@math.tsinghua.edu.cn}

\author{Mathieu Sti\'enon}
\address{Department of Mathematics, Pennsylvania State University}
\email{stienon@math.psu.edu}

\author{Ping Xu}
\address{Department of Mathematics, Pennsylvania State University}
\email{ping@math.psu.edu}

\thanks{Research partially supported by NSF grants DMS0605725;
DMS0801129; and DMS1101827, NSFC grant 11001146, and  THU grant 2012THZ02154.}

\begin{abstract}
We prove a 2-categorical analogue of a classical result of Drinfeld:
there is a one-to-one correspondence between connected, simply-connected
Poisson Lie 2-groups and Lie 2-bialgebras. In fact, we also prove that
there is a one-to-one correspondence between connected, simply connected
quasi-Poisson 2-groups and quasi-Lie 2-bialgebras.
Our approach relies on a ``universal lifting theorem'' for Lie 2-groups:
an isomorphism between the graded Lie algebras of multiplicative polyvector fields on the Lie 2-group on one hand and of polydifferentials on the corresponding Lie 2-algebra on the other hand.
\end{abstract}

\maketitle
\tableofcontents

\section*{Introduction}

A Poisson group is a Lie group equipped with a compatible Poisson structure.
Poisson groups are the classical limit of quantum groups and have been
extensively studied in the past two decades.
For instance, Drinfeld proved that there is a bijection between connected,
simply connected Poisson groups and Lie bialgebras \cite{MR688240,MR934283}.

Lie 2-groups (also called strict Lie 2-groups in the literature)
are Lie group objects in the category of Lie groupoids,
or equivalently Lie groupoid objects in the category of Lie groups.
More explicitly, a Lie 2-group is a Lie groupoid $\Gamma_1\toto \Gamma_0$,
where both $\Gamma_1$ and $\Gamma_0$ are Lie groups and
all the groupoid  structure maps are group homomorphisms.
Lie 2-groups are special instances of Mackenzie's double groupoids \cite{MR1697617}.

The recent categorification trend motivates the search for an appropriate notion
of quantum 2-groups. Poisson 2-groups are a
natural first step in that direction.
By a Poisson 2-group, we mean a Lie 2-group equipped with a Poisson structure
$\gthetaPi$ on $\Gamma_1$, which is multiplicative with respect to
both the group and the  groupoid structures on $\Gamma_1$. In other words,
$(\Gamma_1,\gthetaPi)$ is simultaneously both a Poisson group
\cite{MR688240,MR934283} and a Poisson groupoid \cite{MR959095}.

Lie 2-algebras are Lie algebra objects in the category of Lie algebroids \cite{MR2068522}.
They can be identified with Lie algebra crossed modules:
pairs of Lie algebras $\thetaalgebra$ and $\galgebra$
together with a linear map $\phi:\thetaalgebra\to\galgebra$
and an action of $\galgebra$ on $\thetaalgebra$ by derivations
satisfying a certain compatibility condition.
Likewise, a Lie 2-bialgebra can be considered as a Lie bialgebra crossed module,
i.e.\ a pair of Lie algebra crossed modules in duality:
$\crossedmoduletriple{\thetaalgebra}{\phi}{\galgebra}$ and
$\crossedmoduletriple{\galgebrastar}{-\phi^*}{\thetaalgebrastar}$
are both Lie algebra crossed modules, and
$(\Lietwo{\thetaalgebra}{\galgebra},\Lietwo{\galgebrastar}{\thetaalgebrastar})$
is a Lie bialgebra.

We prove that, at the infinitesimal level, Poisson 2-groups induce Lie 2-bialgebras.
More precisely, we prove the following analogue of Drinfeld's theorem:

\begin{main}\label{firstmain}
There is a one-to-one correspondence between connected, simply-connected
Poisson 2-groups and Lie 2-bialgebras.
\end{main}

We will also prove a more general result:

\begin{main}\label{secondmain}
There is a one-to-one correspondence between connected, simply-connected
quasi-Poisson 2-groups and quasi-Lie 2-bialgebras.
\end{main}

Quasi-Poisson 2-groups are, in a certain sense, the 2-categorical analogues
of Kosmann-Schwarzbach's quasi-Poisson groups \cite{MR1188453}.
A quasi-Poisson 2-group is a Lie 2-group $\Gamma_1\toto \Gamma_0$
endowed with a multiplicative quasi-Poisson structure on $\Gamma_1$,
i.e.\ a multiplicative bivector field $\gthetaPi$ on $\Gamma_1$ such that
the Schouten bracket $[ \gthetaPi, \gthetaPi]$ is some sort of coboundary.

A natural generalization of Lie 2-algebras (or Lie algebra crossed modules),
weak Lie 2-algebras are two-term $L_\infty$ algebras.
They can be described concisely in terms of the shifted degree ``big bracket,''
which is a Gerstenhaber bracket on ${\Sbullet}(\shiftby{V}{2}\oplus{\shiftby{\Vs}{1}})$.
Here $V=\theta\oplus \galgebradegzero$ is a graded vector space,
where $\theta$ is of degree $1$ and $\galgebradegzero$ is of degree 0.
Identifying ${\Sbullet}(\shiftby{V}{2}\oplus{\shiftby{\Vs}{1}})$
with the space $\Gamma(\wedge^\bullet\shiftby{T}{4}M)$
of polyvector fields on $M=\shiftby{\Vs}{-2}$ with polynomial coefficients,
the big bracket can be simply described as the Schouten bracket of polyvector fields
on $M$.

In \cite{2011arXiv1109.2290C}, we developed a notion of weak Lie 2-bialgebras:
objects that are simultaneously weak Lie 2-algebras as well as weak Lie 2-coalgebras,
both structures being compatible with one another in a certain sense.
In terms of the big bracket, a weak Lie 2-bialgebra on a graded vector space $V$
is a degree-$(-4)$ element $t$ of ${\Sbullet}(\shiftby{V}{2}\oplus{\shiftby{\Vs}{1}})$
satisfying $\{t,t\}=0$.
Quasi-Lie 2-bialgebras are a special instance of weak Lie 2-bialgebras.

Our proofs of Theorems~\ref{firstmain} and~\ref{secondmain} rely on the following
``universal lifting theorem,'' which should be of independent interest:

\begin{main}\label{thirdmain}
Given a Lie 2-group $\Gamma_1\toto \Gamma_0$,
if both $\Gamma_1$ and $\Gamma_0$ are connected and simply connected,
then the graded Lie algebras $\bigoplus_{k\geq 0} \strictmultiplicative{k}{\Gamma_1}$
and $\bigoplus_{k\geq 0} \cala_k$ are isomorphic.
\end{main}

Here $\bigoplus_{k\geq 0} \strictmultiplicative{k}{\Gamma_1}$ denotes
the space of multiplicative polyvector fields on $\Gamma_1$ which,
being closed with respect to the Schouten bracket, is naturally a graded Lie algebra.
On the other hand, $\bigoplus_{k\geq 0}\cala_k$ denotes the graded Lie algebra
formed by the polydifferentials on the associated Lie 2-algebra --- the infinitesimal counterparts of the multiplicative polyvector fields on the Lie 2-group.

Theorems~\ref{firstmain} and~\ref{secondmain} are proved simply by expressing the algebraic data
defining the weak Lie 2-bialgebra structure in terms of the graded Lie algebra $\bigoplus_{k\geq 0}\cala_k$.

We refer to  the recent papers \cite{2011arXiv1101.3996L, MR2861783, 2011arXiv1103.5920S} on
integration of  Courant algebroids to symplectic 2-groupoids, which
may have a close connection  to our work.

\subsection*{Acknowledgments} We would like to thank several institutions
for their hospitality while work on this project was being done:
Penn State University (Chen),
Universit\'e du Luxembourg (Chen,  Sti\'enon, and Xu),
Institut des Hautes \'Etudes Scientifiques
and Beijing International Center for Mathematical Research (Xu).
We would also like to thank Anton Alekseev, Benjamin Enriques,
Yvette Kosmann-Schwarzbach, Kirill Mackenzie, Jim Stasheff,
Henrik Strohmayer, and Alan Weinstein
for useful discussions and comments.

\section{Quasi-Lie  2-bialgebras}

In this section, we recall some basic notions regarding quasi Lie 2-bialgebras developed in~\cite{2011arXiv1109.2290C}.

\subsection{The big  bracket}
\label{Schoutenbracket}

We will introduce a graded version  of the big bracket
\cite{MR2103012,MR1046522} involving graded vector spaces.

Let $V=\bigoplus_{k\in\ZZ}\degreesubspace{V}{k}$ be a graded vector
space. Consider the $\ZZ$-graded manifold $M=\shiftby{\Vs}{-2}$ and
the shifted tangent space
\[ \shiftby{T}{4}M\cong \shiftby{(M\times \shiftby{\Vs}{-2})}{4}\cong M\times\shiftby{\Vs}{2} .\]

Consider the space of polyvector fields on $M$ with polynomial coefficients:
\begin{multline*}
\Gamma_{ }(\wedge^\bullet\shiftby{T}{4}M)
\cong\Sbullet(M^*)\otimes\Sbullet(\shiftby{(\shiftby{\Vs}{2})}{-1}) \\
\cong\Sbullet(\shiftby{V}{2})\otimes\Sbullet(\shiftby{\Vs}{1})
\cong{\Sbullet}(\shiftby{V}{2}\oplus{\shiftby{\Vs}{1}}).
\end{multline*}
We write $\symmetricalgebraa$ for ${\Sbullet}(\shiftby{V}{2}\oplus
{\shiftby{\Vs}{1}})$ and $\symmetricproduct$ for the symmetric tensor product in
$\symmetricalgebraa$.

There is a standard way to endow $\symmetricalgebraa=\Gamma
(\wedge^\bullet\shiftby{T}{4}M)$ with a graded Lie bracket, i.e.
 the Schouten bracket,
which is denoted by $\Poissonbracket{\tobefilledin,\tobefilledin}$. It is  a
 bilinear map
 $\Poissonbracket{\tobefilledin,\tobefilledin}:$  $\symmetricalgebraa\otimes
 \symmetricalgebraa\to\symmetricalgebraa$ satisfying the following properties:
\begin{enumerate}
\item $\Poissonbracket{v,v'}=\Poissonbracket{\epsilon,\epsilon'}=0$,
for all $v,v'\in\shiftby{V}{2}$ and $\epsilon,\epsilon'\in \shiftby{\Vs}{1}$;
\item $\Poissonbracket{v,\epsilon}= \minuspower{\abs{v}}\pairing{v}{\epsilon}$, for all
$v\in\shiftby{V}{2}$ and $\epsilon\in \shiftby{\Vs}{1}$;
\item $\Poissonbracket{e_1,e_2}=-\minuspower{ (\abs{e_1}+3)
(\abs{e_2}+3)  }\Poissonbracket{e_2,e_1} $,  for all $e_1,e_2\in \symmetricalgebraa$;
\item $\Poissonbracket{e_1,e_2\symmetricproduct e_3}=
\Poissonbracket{e_1,e_2}\symmetricproduct e_3 + \minuspower{(
\abs{e_1}+3)\abs{e_2}}e_2\symmetricproduct
\Poissonbracket{e_1,e_3}$, for all $e_1,e_2,e_3\in\symmetricalgebraa$.
\end{enumerate}

It is clear that $\Poissonbracket{\tobefilledin,\tobefilledin}$
is of degree 3, 
i.e.\
\[ \abs{\Poissonbracket{e_1,e_2}}=\abs{e_1} +\abs{e_2}+3 ,\]
for all homogeneous
$e_i\in \symmetricalgebraa$, 
and the following graded Jacobi identity holds:
\[ \Poissonbracket{e_1,\Poissonbracket{e_2,e_3}}
=\Poissonbracket{ \Poissonbracket{e_1,e_2},e_3}
+ \minuspower{(\abs{e_1}+3)(
\abs{e_2}+3)}\Poissonbracket{e_2,\Poissonbracket{e_1,e_3}} .\]

Hence $(\symmetricalgebraa,\symmetricproduct,\Poissonbracket{\tobefilledin,\tobefilledin})$
is a Schouten algebra, also known as an odd Poisson algebra, or a Gerstenhaber algebra \cite{MR1958834}.

Due to our degree convention, when $V$ is an ordinary
 vector space considered as a graded vector space concentrated at degree $0$,
the bracket above  is different from the usual
big bracket in the literature \cite{MR2103012}.

\subsection{Quasi-Lie  2-bialgebras}

Following Baez-Crans \cite{MR2068522},
a weak Lie 2-algebra is an  $\Linf$-algebra on the 2-term
graded vector space $V=\theta\oplus \galgebradegzero$, where
 $\theta$ is of degree
$1$ and $\galgebradegzero$ is of degree 0.
Unfolding the $\Linf$-structure, we can  define
a weak Lie 2-algebra as a
 pair of vector spaces $\thetaalgebradegone$ and
$\galgebradegzero$ endowed with the following structures:

\begin{enumerate}
\item a linear map $\phimap$: $\thetaalgebradegone\to\galgebradegzero$;
\item a bilinear skewsymmetric
 map  $\bracketmapbracket{\tobefilledin}{\tobefilledin}$:
 $\galgebradegzero\otimes\galgebradegzero\to
\galgebradegzero$;
\item a bilinear map  $\actionmapaction{\tobefilledin}{\tobefilledin}$:
$\galgebradegzero\otimes\thetaalgebradegone\to\thetaalgebradegone$;
\item a trilinear skewsymmetric map $\homotopymap$: $\galgebradegzero\otimes\galgebradegzero\otimes\galgebradegzero\to\thetaalgebradegone$,
called the homotopy map.
\end{enumerate}

These maps are  required to satisfy  the following compatibility conditions:
for all $w,x,y,z \in \galgebradegzero$ and $u,v\in
\thetaalgebradegone$,
\begin{gather*}
\bracketmapbracket{
\bracketmapbracket{x}{y}}{z}+ \bracketmapbracket{
\bracketmapbracket{y}{z}}{x}+\bracketmapbracket{
\bracketmapbracket{z}{x}}{y}+(\phimap\circ\homotopymap)(x,y,z)=0;
\\
\actionmapaction{y}{(\actionmapaction{x}{u})}
-\actionmapaction{x}{(\actionmapaction{y}{u})} +\actionmapaction{
\bracketmapbracket{x}{y}}{u}+\homotopymap(\phimap(u),x,y)=0;
\\
\actionmapaction{\phimap(u)}{v}+\actionmapaction{\phimap(v)}{u}=0;
\\
\phimap(\actionmapaction{x}{u})=\bracketmapbracket{x}{\phimap(u)};
\end{gather*}
and
\begin{multline*}
-\actionmapaction{w}{\homotopymap(x,y,z)}
-\actionmapaction{y}{\homotopymap(x,z,w)}
+\actionmapaction{z}{\homotopymap(x,y,w)}
+\actionmapaction{x}{\homotopymap(y,z,w)}
\\
=\homotopymap(\bracketmapbracket{x}{y},z,w)-\homotopymap(\bracketmapbracket{x}{z},y,w)
+\homotopymap(\bracketmapbracket{x}{w},y,z)
\\
+\homotopymap(\bracketmapbracket{y}{z},x,w)-\homotopymap(\bracketmapbracket{y}{w},x,z)
+\homotopymap(\bracketmapbracket{z}{w},x,y).
\end{multline*}

If $\homotopymap$ vanishes, we call it a strict Lie 2-algebra, or simply a Lie 2-algebra.

Now consider the degree shifted vector spaces $\shiftby{V}{2}$ and
$\shiftby{\Vs}{1}$.
Under such a degree convention, the degrees of $\galgebradegzero$,
$\thetaalgebradegone$, $\galgebrastardegminusone$ and
$\thetaalgebrastardegminustwo$ are specified as follows:
\begin{center}
\begin{tabular}{   l  | l |  l | l | l }
space & $\galgebradegzero$ & $\thetaalgebradegone$ & $\galgebrastardegminusone$
& $\thetaalgebrastardegminustwo$ \\\hline
degree & $-2$&$-1$&$-1$& $-2$
\end{tabular}
\end{center}

We will maintain this convention throughout this section.
We remind the reader that the abbreviation $\Sbullet$
stands for $\Sbullet({\shiftby{\Vs}{1}\oplus \shiftby{V}{2}})$.

\begin{proposition}[\cite{2011arXiv1109.2290C}]
\label{Prop:Lie2algebraelements}
Under the above degree convention, a weak Lie 2-algebra
structure is equivalent to a solution to the equation
\begin{equation}
\label{eq:ss}
\Poissonbracket{s,s}=0,
\end{equation}
where $s=\derivesvarphi+\derivesbracket+
\derivesaction+\derivesh$ is an element
 in $\degreesubspace{\symmetricalgebra}{-4}$ such that
\begin{equation}\label{eqt:derivesdata}
\left\{
\begin{aligned}
\derivesvarphi & \in \thetaalgebrastardegminustwo \symmetricproduct\galgebradegzero,\\
\derivesbracket & \in (\symmetricproduct^2\galgebrastardegminusone)\symmetricproduct \galgebradegzero,\\
\derivesaction & \in\galgebrastardegminusone\symmetricproduct\thetaalgebrastardegminustwo
\symmetricproduct\thetaalgebradegone,\\
\derivesh & \in(\symmetricproduct^3\galgebrastardegminusone)\symmetricproduct\thetaalgebradegone.
\end{aligned}
\right.
\end{equation}
Here the bracket in Eq.~\eqref{eq:ss}
stands for the big bracket as in Section~\ref{Schoutenbracket}.
\end{proposition}

In the sequel, we denote a weak Lie 2-algebra by
$(\thetaalgebradegone{{\to}}\galgebradegzero,s)$ in order to emphasize
the map from $\thetaalgebradegone$ to $\galgebradegzero$.
Sometimes,
 we will omit $s$ and denote a weak Lie 2-algebra simply by
$(\thetaalgebradegone{{\to}}\galgebradegzero)$.
If $\crossedmoduletriple{\galgebrastardegminusone}{ }{\thetaalgebrastardegminustwo}$
is a weak Lie 2-algebra, then $\crossedmoduletriple{\thetaalgebra}{ }{\galgebra}$
is called a weak Lie 2-coalgebra.
Equivalently, a weak Lie 2-coalgebra is a 2-term $\Linf$-structure on $\galgebrastardegminusone\oplus\thetaalgebrastardegminustwo$,
where $\galgebrastardegminusone$ has degree $1$ and
$\thetaalgebrastardegminustwo$ has degree $0$.

Similarly, we have the following

\begin{proposition}[\cite{2011arXiv1109.2290C}]
\label{Prop:Lie2coalgebraelements}
A weak Lie 2-coalgebra  is equivalent to a solution
to the equation
\begin{equation}\nonumber
\Poissonbracket{c,c}=0,
\end{equation}
where $c=\derivescovarphi+\derivescobracket+ \derivescoaction+\derivescoh\in
\degreesubspace{\symmetricalgebra}{-4}$ such that
\begin{equation}\label{eqt:derivescodata}
 \left\{
\begin{aligned}
\derivescovarphi & \in\thetaalgebrastardegminustwo \symmetricproduct\galgebradegzero, \\
\derivescobracket & \in\thetaalgebrastardegminustwo \symmetricproduct (\symmetricproduct^2\thetaalgebradegone), \\
\derivescoaction & \in\galgebrastardegminusone\symmetricproduct\galgebradegzero\symmetricproduct\thetaalgebradegone, \\
\derivescoh & \in\galgebrastardegminusone\symmetricproduct(\symmetricproduct^3\thetaalgebradegone).
\end{aligned}
\right.
\end{equation}
\end{proposition}

We denote such a weak Lie 2-coalgebra by
$(\thetaalgebradegone{{\to}}\galgebradegzero,c)$.

\begin{definition}\label{Defn:Lie2bialgebra}
A weak Lie 2-bialgebra consists of a pair of vector spaces
$\theta$ and $\galgebradegzero$ together with a solution
$t=\derivesbracket+\derivesaction+\derivesh+\derivescovarphi+\derivescobracket
+\derivescoaction+\derivescoh\in\degreesubspace{\symmetricalgebra}{-4}$
to the equation
$\{t, t\}=0$.
Here $\derivesbracket,\derivesaction,\derivesh,\derivescovarphi,\derivescobracket,\derivescoaction,\derivescoh$
are as in  Eqs.~\eqref{eqt:derivesdata} and~\eqref{eqt:derivescodata}.
\newline
If, moreover, $\derivesh=0$, it is called a quasi-Lie 2-bialgebra.
If both $\derivesh$ and $\derivescoh$ vanish, we say that the Lie 2-bialgebra is strict, or simply a Lie 2-bialgebra.
\end{definition}

\begin{proposition}\label{Prop:tt=0impliesll=0andcc=0}
Let $(\thetaalgebradegone, \galgebradegzero, t)$ be a weak
Lie 2-bialgebra as in Definition \ref{Defn:Lie2bialgebra}.
Then
$(\thetaalgebradegone{{\to}}\galgebradegzero,l)$,
where $l=\derivescovarphi+\derivesbracket+ \derivesaction+\derivesh$,
is a weak Lie 2-algebra, while $(\thetaalgebradegone{{\to}}\galgebradegzero,c)$,
where $c=\derivescovarphi+\derivescobracket+ \derivescoaction+\derivescoh$,
is a weak Lie 2-coalgebra.
\end{proposition}

\begin{example}
Assume that $\galgebra$ is a semisimple Lie algebra. Let
$\ipG{\tobefilledin}{\tobefilledin}$ be its Killing form.
Then $\homotopymap(x,y,z)=\plankconstant\ipG{x}{\ba{y}{z}}$, for all $x,y,z\in\galgebra$,
is a Lie algebra $3$-cocycle, where $\plankconstant$ is a constant.
Let $\thetaalgebra=\reals$. Then the trivial map $\reals\to\galgebra$
together with $\homotopymap$ becomes a weak Lie 2-algebra, called
the string Lie 2-algebra \cite{MR2068522}. More precisely, the
string Lie 2-algebra is as follows:
\begin{enumerate}
\item $\thetaalgebra$ is the abelian Lie algebra $\reals$;
\item $\galgebra$ is a semisimple Lie algebra;
\item $\phimap:\thetaalgebra\to\galgebra$ is the trivial map;
\item the action map $\moduleaction:\galgebra\otimes\thetaalgebra\to\thetaalgebra$ is the trivial map;
\item $\homotopymap:\wedge^3\galgebra\to\thetaalgebra$ is given by the map
$\plankconstant\ipG{\cdot}{\ba{\cdot}{\cdot}}$, where $\plankconstant$ is a fixed constant.
\end{enumerate}
Now fix an element $x\in \galgebra$.
We endow $\reals\to \galgebra$ with a weak Lie 2-coalgebra structure as follows:
\begin{enumerate}
\item $\galgebrastar$ is an  abelian Lie algebra;
\item $\thetaalgebrastar\cong\reals$ is an  abelian Lie algebra;
\item $\phi^*:\galgebrastar\to\thetaalgebrastar$ is the trivial map;
\item the $\thetaalgebrastar$-action on $\galgebrastar$ is given by
$\unit\moduleaction\xi=\adjoint{x}^*\xi$, for all $\xi\in\galgebrastar$;
\item $\cohomotopymap:\wedge^3\thetaalgebrastar\to\galgebrastar$ is the trivial map.
\end{enumerate}
One can verify directly that these relations indeed define a weak Lie 2-bialgebra.
\end{example}

\subsection{Lie bialgebra crossed modules}

\begin{definition}
A Lie algebra crossed module consists of a pair of Lie algebras
$\thetaalgebra$ and $\galgebra$, a linear map
$\phi:\thetaalgebra\to\galgebra$, and an action of $\galgebra$ on
$\thetaalgebra$ by derivations satisfying, for all $x,y\in\galgebra$,
$u,v\in\thetaalgebra$,
\begin{enumerate}
\item \label{aun} ${\phi (u)} \moduleaction v=\ba{u}{v}$;
\item \label{adeux} $\phi (  x \moduleaction u)=\ba{ x }{\phi (u)}$,
\end{enumerate}
\end{definition}
where $\moduleaction$ denotes the $\galgebra$-action on $\thetaalgebra$.

Note that \ref{aun} and \ref{adeux}  imply that $\phi$ must be  a Lie algebra homomorphism.
We write $\crossedmoduletriple{\thetaalgebra}{\phi}{\galgebra}$
to denote a Lie algebra crossed module.
The associated semidirect product Lie algebra
is denoted by $\galgebra\ltimes\thetaalgebra$.

The following proposition indicates that crossed modules of Lie algebras
are in one-to-one correspondence with Lie 2-algebras.
We refer the reader to~\cite{MR2068522} for details.

\begin{proposition}
Lie algebra crossed modules are equivalent to (strict) Lie 2-algebras.
\end{proposition}

\begin{definition}\label{Defn:bi-crossedmoduleofLiealgebra}
A Lie bialgebra crossed module is a pair of Lie algebra crossed modules in duality:
$\crossedmoduletriple{\thetaalgebra}{\phi}{\galgebra}$ and
$\crossedmoduletriple{\galgebrastar}{\phiUprotate}{\thetaalgebrastar}$, where $\phiUprotate=-\phi^*$,
are both Lie algebra crossed modules such that
$(\Lietwo{\thetaalgebra}{\galgebra},\Lietwo{\galgebrastar}{\thetaalgebrastar})$ is a Lie bialgebra.
\end{definition}

Lie bialgebra crossed modules are symmetric as we see in the next

\begin{proposition}
If $(\crossedmoduletriple{\thetaalgebra}{\phi}{\galgebra},
\crossedmoduletriple{\galgebrastar}{\phiUprotate
}{\thetaalgebrastar})$ is a Lie bialgebra crossed module,
so is $(\crossedmoduletriple{\galgebrastar}{\phiUprotate
}{\thetaalgebrastar}, \crossedmoduletriple{\thetaalgebra}{\phi}{\galgebra})$.
\end{proposition}

The following result  justifies our terminology.

\begin{proposition}
\label{Prop:liebicrossedmoduleimpliesliebi}
If $(\crossedmoduletriple{\thetaalgebra}{\phi}{\galgebra}$,
$\crossedmoduletriple{\galgebrastar}{\phiUprotate
}{\thetaalgebrastar} )$ is a Lie bialgebra crossed module,
then both pairs $(\thetaalgebra , \thetaalgebrastar)$ and
$(\galgebra, \galgebrastar)$ are Lie bialgebras.
\end{proposition}

\begin{example}
One can  construct a Lie bialgebra crossed module from
an ordinary   Lie bialgebra as follows.
 Given a Lie bialgebra $(\thetaalgebra,\thetaalgebrastar)$,
consider  the trivial Lie algebra crossed module
 $\crossedmoduletriple{\thetaalgebra}{\Id}{\thetaalgebra}$,
 where  the second $\thetaalgebra$ acts on the first
$\thetaalgebra$ by the adjoint action.
In the mean time,  consider the dual Lie algebra crossed module
$\crossedmoduletriple{\thetaalgebrastar}{-\Id }{ \thetaalgebrastar}$,
where the second
$\thetaalgebrastar$ is equipped with the opposite Lie bracket:
$-\bas{\tobefilledin}{\tobefilledin}$, and the action of the second
$\thetaalgebrastar$ on the first $\thetaalgebrastar$ is given by
$\kappa_2\moduleaction{\kappa_1}= -\bas{\kappa_2}{\kappa_1}$,
for all $\kappa_1,\kappa_2\in\thetaalgebrastar$.
It is simple to see that
$(\crossedmoduletriple{\thetaalgebra}{\Id}{\thetaalgebra},
\crossedmoduletriple{\thetaalgebrastar}{-\Id }{ \thetaalgebrastar})$
is indeed  a Lie bialgebra crossed module.
\end{example}

The following theorem was proved in~\cite{2011arXiv1109.2290C}.

\begin{theorem}
\label{Thm:Liebialgebracm1-1strictLie2bialgebra}
There is a bijection between Lie bialgebra crossed modules
and (strict) Lie 2-bialgebras.
\end{theorem}

\begin{example}
Consider the Lie subalgebra $\mathrm{u}(n)\subset\mathfrak{gl}_n(\CC)$
of $n\times n$ skew-Hermitian matrices.
Let ${\thetaalgebra}\subset \mathfrak{gl}_n(\CC) $
be the Lie subalgebra consisting  of
upper triangular matrices whose diagonal elements are real numbers.
It is standard that $(\thetaalgebra,\mathrm{u}(n))$ is a Lie bialgebra.
Indeed $\thetaalgebra\oplus\mathrm{u}(n)\cong \mathfrak{gl}_n(\CC)$,
and both $\thetaalgebra$ and $\mathrm{u}(n)$ are Lagrangian
subalgebras of $\mathfrak{gl}_n(\CC)$ under the nondegenerate
pairing $\pairing{X}{Y}=\mathrm{Im}(\mathrm{Tr}(XY))$, for $X,Y\in \mathfrak{gl}_n(\CC)$.
Hence $({\thetaalgebra},{\mathrm{u}(n)}, \mathfrak{gl}_n(\CC) )$
is a Manin triple,  and thus $(\thetaalgebra,\mathrm{u}(n))$  forms a Lie bialgebra.
\newline
Let $\galgebra$ denote the Lie algebra of traceless upper triangular matrices
with real diagonal coefficients.
It turns out that $\crossedmoduletriple{\thetaalgebra}{\phi}{\galgebra}$,
where $\phi$ is the map $A\mapsto A-\text{tr}A$, is a Lie bialgebra crossed module.
\end{example}

\section{Universal lifting theorem}

\subsection{Lie 2-groups}

A Lie 2-group (also called strict Lie 2-groups in the literature) is a Lie groupoid $\Gamma_1\toto\Gamma_0$,
where both $\Gamma_1$ and $\Gamma_0$ are Lie groups and
all the groupoid  structure maps are group homomorphisms.
A Lie 2-group  is a special case of double groupoid in the sense
of Mackenzie~\cite{MR1697617}.

\begin{definition}[\cite{MR0017537,MR0030760}]
A Lie group crossed module consists of a Lie group homomorphism $\Phi:\thetagroup\to\ggroup$
and an action of $\ggroup$ on $\thetagroup$ by automorphisms satisfying the following compatibility conditions:
\begin{enumerate}
\item \label{bun} ${\Phi(\alpha)}\groupaction\beta=\alpha\beta\inverse{\alpha}$;
\item \label{bdeux} $\Phi( g\groupaction\beta )=g\Phi(\beta)\inverse{g}$,
\end{enumerate}
for all $g\in\ggroup$ and $\alpha,\beta\in\thetagroup$. Here
$g\groupaction\beta $ denotes the action of
$g\in\ggroup$ on $\beta \in\thetagroup$.
\end{definition}

We write $\groupcrossedmoduletriple{\thetagroup}{\Phi}{\ggroup}$ to denote a Lie group crossed module.

\begin{proposition}
\label{thm:2groupstructuredetails}
There is a bijection between Lie 2-groups and crossed modules
of Lie groups.
\end{proposition}

\begin{proof}
This is standard. For instance, see~\cite{MR2068521,MR0440553,MR0419643,MR1001474,MR2280287}.
Here we will  sketch the construction of the Lie 2-group out of a crossed module, which will be needed later on.

The Lie 2-group corresponding to a Lie group crossed module
$\groupcrossedmoduletriple{\thetagroup}{\Phi}{\ggroup}$ will be denoted by $\crossedmoduletwogroup\toto G$,
or simply $\crossedmoduletwogroup$, by abuse of notations.
Here the group structure on $\crossedmoduletwogroup $ is as follows:
\begin{itemize}
\item group multiplication:
$(g,\alpha)\gpmulti(h,\beta)=(gh,({\inverse{h}}\groupaction\alpha)\beta)$;
\item group unit:
$\unit_{\gpmulti}=(\unit_{\ggroup},\unit_{\thetagroup})$, here
$\unit_{\ggroup}$ and $\unit_{\thetagroup}$ denote, respectively, the
group unit elements of $\ggroup$ and $\thetagroup$;
\item group inversion:
$\gpinverse{(g,\alpha)}=(\inverse{g},\inverse{(g\groupaction\alpha)})$.
\end{itemize}
The groupoid structure on $\crossedmoduletwogroup\rightrightarrows \ggroup$ is as follows:
\begin{itemize}
\item source and target maps:
$\source(g,\alpha)=g$, $\target(g,\alpha)=g\Phi(\alpha)$;
\item groupoid multiplication:
$(g,\alpha)\gpoidmulti(h,\beta)=(g,\alpha\beta)$, if
$h=g\Phi(\alpha)$;
\item groupoid units: $(g,\unit_{\thetagroup})$;
\item groupoid inversion:
$\gpoidinverse{(g,\alpha)}=(g\Phi(\alpha),\inverse{ \alpha})$.
\end{itemize}
\end{proof}

In the sequel, we will use Lie 2-groups and crossed modules of
Lie groups interchangeably.

\subsection{Multiplicative polyvector fields on Lie groupoids}

We recall some standard results regarding multiplicative
polyvector fields on a Lie groupoid.
Let $\Gamma\toto M$ be a Lie groupoid with   source and target maps
$s$ and  $t$, respectively.
Consider the graph of the groupoid multiplication
$\Lambda=\set{(p,q,pq)|t(p)=s(q)}$,
which is a submanifold  in $\Gamma\times\Gamma\times\Gamma$.

Recall that a $k$-vector field $\Sigma\in
\polyvectorfields{k}{\Gamma}$ is said to be multiplicative if
$\Lambda$ is coisotropic with respect to
$\Sigma\times\Sigma\times\minuspower{k+1}\Sigma$ \cite{MR2911881}. In other words,
\[ (\Sigma\times\Sigma\times\minuspower{k+1}\Sigma)(\xi_1,\cdots,\xi_k)=0 ,\quad\forall \xi_1,\cdots,\xi_k\in \Lambda^{\perp} \]
for all $\xi_1,\cdots,\xi_k\in \Lambda^{\perp}$, where
\[ \Lambda^{\perp}=\set{\xi\in T^*_{\lambda}(\Gamma\times\Gamma\times\Gamma)\text{ s.t. } \lambda\in
\Lambda,\pairing{\xi}{v}=0, \forall v\in T_\lambda \Lambda} .\]

A $k$-vector field $\Sigma$ on $\Gamma$ is said to be affine
if $\ba{\Sigma}{\overleftarrow{X}}$
is left invariant for all $X\in\sections{A}$. Here
$A$ denotes the  Lie algebroid of $\Gamma$, and $\overleftarrow{X}$
denotes the left invariant vector field on $\Gamma$ corresponding to $X$.

The following lemma gives a  useful characterization of  multiplicative polyvector fields.

\begin{lemma}[Theorem~2.19 in~\cite{MR2911881}]
\label{Lem:Sigmagroupoidmultiplicativeiff}
A $k$-vector field $\Sigma$ is multiplicative if and only if the following three conditions hold:
\begin{enumerate}
\item \label{cun} $\Sigma$ is affine;
\item \label{cdeux} $M$ is a coisotropic submanifold of $\Gamma$;
\item \label{ctrois} for any $\xi\in\Omega^1(M)$, $\inserts_{\target^*(\xi)}\Sigma$ is left invariant.
\end{enumerate}
\end{lemma}

\begin{remark}
The  statement of Theorem~2.19 in~\cite{MR2911881}
contains more conditions but some of them are redundant.
\end{remark}

Along the base manifold $M$, the tangent bundle $T\Gamma$ admits
a natural decomposition \[ T \Gamma|_M= TM\oplus A ,\]
where $A$ is identified with $T^s \Gamma|_M$,
the tangent bundle to the $s$-fibers along $M$.
Denote by $\rho:A\to TM$ the anchor map. Then
$\rho$ is equal to $t_*: T^s \Gamma|_M\to TM$.

Let $Z_k$ be the set of all elements $w$ of $TM\wedge(\wedge^{k-1}A)$ satisfying
\[ \inserts_{\zeta_1}\inserts_{\rho^*\zeta_2}w
=-\inserts_{\zeta_2}\inserts_{\rho^*\zeta_1}w,
\quad\forall\zeta_1,\zeta_2\in T^*M .\]

Let $\rhopush$ be a degree-$0$ derivation of $\Gamma (\wedge^\bullet (TM\oplus A))$
such that $\rhopush(a+b)=\rho(a)$, for all $a\in A$ and $b\in TM$.

\begin{lemma}\label{Lem:skewdeltaderive1}
For any $w\in Z_k$ and $j\geq 1$, we have
\begin{equation}\label{Eqn:skewdeltaiteratedformula2}
\inserts_{\rho^*\zeta}(\rhopush^{j-1}w)
=\rhopush^j(\inserts_{\zeta}w)
=\frac{1}{j+1}\inserts_{\zeta}(\rhopush^j w), \quad\forall\zeta\in T^*M.
\end{equation}
\end{lemma}

\begin{proof}
First, note that we have the following identities:
\begin{gather}
\label{eq:45a}
\inserts_{\zeta}\circ\rhopush-\rhopush\circ\inserts_{\zeta}=\inserts_{\rho^*\zeta},
\\ \label{Eqn:insertsphiUPstarphipusha}
\inserts_{\rho^*\zeta}\circ\rhopush=\rhopush\circ\inserts_{\rho^*\zeta},
\end{gather}
where $\zeta\in T^*M$, and both sides of Eqs.~\eqref{eq:45a} and~\eqref{Eqn:insertsphiUPstarphipusha}
are considered as linear maps $\wedge^\bullet(TM\oplus A)\to\wedge^{\bullet-1}(TM\oplus A)$.

Now we prove Eq.~\eqref{Eqn:skewdeltaiteratedformula2} by induction.
If $j=1$, the equation
\[ \inserts_{\rho^*\zeta} w = \rhopush (\inserts_{\zeta}w) \]
follows from the definition of $Z_k$.
By Eq.~\eqref{eq:45a}, we have
\[ (\inserts_{\zeta}\circ\rhopush-\rhopush\circ\inserts_{\zeta})w
=\inserts_{\rho^*\zeta}w=\rhopush (\inserts_{\zeta}w) .\]
It thus follows that
\[ \rhopush (\inserts_{\zeta}w)=\frac{1}{2}\inserts_{\zeta}(\rhopush w) .\]

Assume that Eq.~\eqref{Eqn:skewdeltaiteratedformula2} is valid for $j\geq 1$.
Then, using Eq.~\eqref{Eqn:insertsphiUPstarphipusha}, we have
\[ \inserts_{\rho^*\zeta}(\rhopush^{j}w)= (\rhopush\circ
\inserts_{\rho^*\zeta})(\rhopush^{j-1}w) =(\rhopush\circ\rhopush
^j )(\inserts_{\zeta}w)=\rhopush^{j+1}(\inserts_{\zeta}w) .\]

Moreover, using Eq.~\eqref{eq:45a}, we have
\begin{multline*}
\rhopush^{j+1}(\inserts_{\zeta}w)=\rhopush (\rhopush
^j(\inserts_{\zeta}w))=
\frac{1}{j+1}(\rhopush\circ\inserts_{\zeta}\circ\rhopush^j) w \\
=\frac{1}{j+1}(\inserts_{\zeta}\circ
\rhopush-\inserts_{\rho^*\zeta})\circ\rhopush^j
w=\frac{1}{j+1}\inserts_{\zeta} ( \rhopush^{j+1}w)-
\rhopush^{j+1}(\inserts_{\zeta}w)\bigr),
\end{multline*}
which implies that
\[ \rhopush^{j+1}(\inserts_{\zeta}w)=
\frac{1}{j+2}\inserts_{\zeta}(\rhopush^{j+1} w) .\]
\end{proof}

\begin{proposition}\label{Thm:SigmaalongMdecomposition}
Given a multiplicative $k$-vector field $\Sigma$ on $\Gamma$,
there exists a section $\sigma\in\sections{TM\wedge(\wedge^{k-1}A)}$ such that
\begin{equation}\label{Eqn:Sigmaatbasepowerseries}
\Sigma|_M = \frac{\Id-e^{-\rhopush}}{\rhopush} (\sigma)
= \sigma-\tfrac{1}{2!}\rhopush\sigma+\tfrac{1}{3!}\rhopush^2
\sigma+\cdots-\tfrac{\minuspower{k}}{k!}\rhopush^{k-1}\sigma
.\end{equation}
Moreover, $\sigma$ satisfies the following properties:
\begin{gather}\label{differentialonfunctions}
\partial_{\Sigma}(f)=\minuspower{k-1}\inserts_{\mathrm{d}f}\sigma,
\quad\forall f\in\CIM, \\ \label{symmetricityofvarpi}
\inserts_{\zeta}\inserts_{\rho^*\xi}\sigma
=-\inserts_{\xi}\inserts_{\rho^*\zeta}\sigma,
\quad\forall\xi,\zeta\in\Omega^1(M). \nonumber
\end{gather}
\end{proposition}

\begin{proof}
Since $M$ is  coisotropic in $\Gamma$ with respect to $\Sigma$, we may write
\begin{equation}\label{SigmaatMdecompose}
\Sigma|_M=\sigma^{1,k-1}+\sigma^{2,k-2}+\cdots+\sigma^{k,0},
\end{equation}
where $\sigma^{i,k-i}\in\sections{(\wedge^i TM)\wedge(\wedge^{k-i}A)}$.

Also observe that, for any $1$-form $\xi\in\Omega^1(M)$, we have
\begin{equation}\label{Eqn:targetstarxi}
\target^*(\xi)|_M=\xi+\rho^*\xi\in\sections{T^*M \oplus\As},
\end{equation}
where $T^*\Gamma|_M$ is naturally identified with $T^*M \oplus A^*$.
On the other hand, Condition~\ref{ctrois} of Lemma~\ref{Lem:Sigmagroupoidmultiplicativeiff}
implies that $(\inserts_{\target^*(\xi)}\Sigma)|_M$ is tangent
to the $\source$-fibers, and therefore contains only $\Gamma (\wedge^kA)$-components.
Using Eqs.~\eqref{SigmaatMdecompose} and~\eqref{Eqn:targetstarxi}, we obtain
\begin{equation}\label{temp2}
\left\{
\begin{aligned}
\inserts_{\rho^*\xi}&\sigma^{1,k-1} = -\inserts_{\xi}\sigma^{2,k-2}\\
\inserts_{\rho^*\xi}&\sigma^{2,k-2} = -\inserts_{\xi}\sigma^{3,k-3}\\
&\vdots  \\
\inserts_{\rho^*\xi}&\sigma^{k-1,1} = -\inserts_{\xi}\sigma^{k,0}.
\end{aligned}
\right.
\end{equation}

Note that, for all $f\in\CIM$,
\begin{equation}
\label{eq:partialf}
\partial_{\Sigma}(f)=\ba{\Sigma}{\target^*f}|_M=\minuspower{k+1}\inserts_{\target^*(\mathrm{d}f)}\Sigma|_M
=\minuspower{k+1}\inserts_{\mathrm{d}f}\sigma^{1,k-1}.
\end{equation}

Since
\[ 0=\partial_{\Sigma}\ba{f_1}{f_2}=\ba{\partial_{\Sigma}(f_1)}{f_2}+\minuspower{k-1}\ba{f_1}{\partial_{\Sigma}(f_2)} ,\]
it follows that
\begin{equation}
\label{eq:sigma12}
\inserts_{\rho^*\mathrm{d}f_2}\inserts_{\mathrm{d}f_1}\sigma^{1,k-1}+
\inserts_{\rho^*\mathrm{d}f_1}\inserts_{\mathrm{d}f_2}\sigma^{1,k-1}=0
.\end{equation}

Let $\sigma=\sigma^{1,k-1}$. We will prove the following identity by induction on $i$:
\begin{equation}\label{Eqt:sigmainduction}
\sigma^{i,k-i}=\frac{\minuspower{i-1}}{i!}\rhopush^{i-1} \sigma.
\end{equation}

The case $i=1$ is obvious. Assume that Eq.~\eqref{Eqt:sigmainduction} is valid for $i$.
Then, by Eq.~\eqref{temp2},
\begin{align*}
\inserts_{\xi}\sigma^{i+1,k-i-1}&=-\inserts_{\rho^*\xi}\sigma^{i,k-i} &&\text{(by the induction assumption)} \\
&=\frac{\minuspower{i }}{i!}\inserts_{\rho^*\xi}\rhopush^{i-1} \sigma &&\text{(by Eq.~\eqref{Eqn:skewdeltaiteratedformula2})} \\
&=\frac{\minuspower{i }}{i!}\frac{1}{i+1}\inserts_{ \xi}\rhopush^{i } \sigma && \\
&=\inserts_{ \xi}\bigl(\frac{\minuspower{i }}{(i+1)!}\rhopush^{i } \sigma \bigr) . &&
\end{align*}
Thus Eq.~\eqref{Eqt:sigmainduction} is proved.
Therefore, Eq.~\eqref{SigmaatMdecompose} implies Eq.~\eqref{Eqn:Sigmaatbasepowerseries},
Eq.~\eqref{eq:partialf} implies Eq.~\eqref{differentialonfunctions},
and Eq.~\eqref{eq:sigma12} implies  Eq.~\eqref{symmetricityofvarpi}.
\end{proof}

\subsection{$k$-differentials of a  Lie algebroid}

It is known that the Schouten bracket of two multiplicative polyvector fields
on a Lie groupoid
is still multiplicative. Therefore,
 the space of multiplicative polyvector fields is a graded Lie algebra
 \cite{MR2911881}.
On the level of Lie algebroids, multiplicative $k$-vector fields
correspond to $k$-differentials of  the  Lie algebroid, whose definition
we recall below.

Given a Lie algebroid $A$, a $k$-differential is a linear map
\[ \partial:\sections{\wedge^\bullet A}\to\sections{\wedge^{\bullet+k-1}A} \]
satisfying
\begin{align*}
& \partial(P\wedge Q)=(\partial P)\wedge Q
+\minuspower{\abs{P}(k-1)}P\wedge (\partial Q),\\
& \partial\ba{P}{Q}=\ba{\partial
P}{Q}+\minuspower{(\abs{P}-1)(k-1)}\ba{P}{\partial Q},
\end{align*}
for all $P,Q\in\sections{\wedge^\bullet A}$.

The commutator of a $k_1$differential $\partial_1$ and a $k_2$-differential $\partial_2$ is the
$(k_1+k_2-1)$-differential
\[ \gradedcommutator{\partial_1}{\partial_2}=\partial_1\circ\partial_2-\minuspower{(k_1-1)(k_2-1)}\partial_2\circ\partial_1 .\]

\begin{theorem}[\cite{MR2911881}]\label{Thm:David-multi-diff}
Let $\Gamma$ be a Lie groupoid, and $A$ its Lie algebroid. Then
every multiplicative $k$-vector field $\Sigma$ induces a
$k$-differential $\partial_{\Sigma}$ by
$\overleftarrow{\partial_{\Sigma}(P)}=\ba{\Sigma}{\overleftarrow{P}}$,
for all $P\in\sections{\wedge^\bullet A}$. Here
$\overleftarrow{V}$ denotes the left invariant polyvector field on
$\Gamma$ determined by $V$.
\newline
Moreover, the map $\Sigma\mapsto\partial_{\Sigma}$ is a homomorphism of graded Lie algebras,
which is an isomorphism of graded Lie algebras provided $\Gamma$ is $s$-connected and $s$-simply connected.
\end{theorem}

In  case of Lie groups, $k$-differentials of multiplicative $k$-vector fields
can be  described more explicitly.

\begin{lemma}\label{Lem:partialSigmacalculate}
Let $\Sigma$ be a multiplicative $k$-vector field on a Lie group $G$.
\begin{enumerate}
\item The map $\intsigma:G\to\wedge^k\galgebra$ defined by
$\intsigma(g)=L_{\inverse{g}*}(\Sigma|_{g})$ is a Lie group $1$-cocycle.
\item The Lie algebra $1$-cocycle induced by $\intsigma$ is the
$k$-differential $\partial_{\Sigma}$:
\begin{equation*}
\left.\frac{d}{dt}\right|_{t=0}(\intsigma|_{\exp tx})
=\left.\frac{d}{dt}\right|_{t=0} L_{{\inverse{\exp}tx}*}(\Sigma|_{\exp tx})
=-\partial_{\Sigma}(x) ,\quad\forall x\in\galgebra
.\end{equation*}
\end{enumerate}
\end{lemma}

\subsection{Infinitesimal data of multiplicative vector fields on  Lie 2-groups}

This section is devoted to the description of infinitesimal
data of multiplicative vector fields on Lie 2-groups.

\subsubsection{Multiplicative polyvector fields on Lie 2-groups}

Let $\groupcrossedmoduletriple{\thetagroup}{\Phi}{ \ggroup}$ be a
Lie group crossed module and $\crossedmoduletwogroup$ the
corresponding  Lie 2-group. Consider a $k$-vector field
$\gthetaV\in \polyvectorfields{k}{\crossedmoduletwogroup}$.
\begin{definition}
A multiplicative $0$-vector field on $\crossedmoduletwogroup$ is a  smooth function
$f\in \cinf{\crossedmoduletwogroup}$ subject to the following
conditions:
\begin{gather*}
f(p\gpmulti q)=f(p)+f(q), \quad \forall p,q \in\crossedmoduletwogroup; \\
f(p\gpoidmulti q)=f(p)+f(q), \quad \forall p,q \in \crossedmoduletwogroup \text{ s.t. } \target(p)=\source(q) .
\end{gather*}
For $k\geq 1$, a $k$-vector field $\gthetaV$ is called  multiplicative if
it is multiplicative with respect to both the group and  the
groupoid structure on $\crossedmoduletwogroup$. In other words, the graph of
the group multiplication
\[ \Lambda^{gp}=\set{(r_1, r_2, r_1 \gpmulti r_2)| r_1, r_2\in\crossedmoduletwogroup} \]
and the graph of  the groupoid multiplication
\[ \Lambda^{gpd} =\set{(r_1, r_2, r_1 \gpoidmulti r_2)| r_1, r_2\in\crossedmoduletwogroup,\target(r_1)=\source(r_2)} \]
are both coisotropic with respect to the $k$-vector field
$(\gthetaV,\gthetaV,(-1)^{k+1}\gthetaV)$ on
$(\crossedmoduletwogroup)\times(\crossedmoduletwogroup)\times(\crossedmoduletwogroup)$.
\end{definition}

Denote the space of  multiplicative $k$-vector fields by
$\strictmultiplicative{k}{\crossedmoduletwogroup}$.
The following lemma follows immediately.

\begin{lemma}
When endowed with the Schouten bracket, the space of multiplicative polyvector fields
\[ \strictmultiplicative{\bullet}{\crossedmoduletwogroup}:=\oplus_{k\geq 0}\strictmultiplicative{k}{\crossedmoduletwogroup} \]
is a graded Lie algebra.
\end{lemma}

\begin{remark}
It is easy to see that $f\in \strictmultiplicative{0}{\crossedmoduletwogroup}$
if and only if $f(g,\alpha)=\nu(\alpha)$, $\forall \alpha\in\thetagroup$
 and $g\in\ggroup$, where $\nu\in C^\infty (\thetagroup )$
satisfies $\nu|_{\alpha\beta}=\nu|_{\alpha}+\nu|_{\beta}$ and
$\nu|_{g\groupaction \alpha}=\nu|_{\alpha}$, $\forall \alpha,\beta\in\thetagroup$.
\end{remark}

\subsubsection{The infinitesimal   data}

It is natural to ask what is the infinitesimal   data of a
multiplicative $k$-vector field (with $k\geq 1$) on a Lie 2-group.
To answer this question, we need, as a first step, to describe the Lie algebroid $\LiealgebroidoverG$
of the groupoid $\crossedmoduletwogroup\rightrightarrows \ggroup$.

It is simple to see that
$ \LiealgebroidoverG$ is  the transformation Lie algebroid
$\ggroup\rtimes \thetaalgebra \to G$,
where the $\thetaalgebra$-action on $G$ is
$ u\mapsto \overleftarrow{ \phi(u)}, \forall~~ u\in\thetaalgebra$.
Here the superscript $\overleftarrow{ \tobefilledin}$ stands for the left invariant
vector field on $\ggroup$ associated to a Lie algebra   element in
$\galgebra$.
It follows from Theorem \ref{Thm:David-multi-diff}
 that a multiplicative
$k$-vector field $\gthetaV \in
\strictmultiplicative{k}{\crossedmoduletwogroup}$ induces a
$k$-differential
\begin{equation}\label{Eqt:defofgroupoiddifferential}
\groupoiddifferential{}:\sections{\wedge^\bullet\LiealgebroidoverG}\to\sections{\wedge^{\bullet+k-1}\LiealgebroidoverG},
\end{equation}
of the Lie algebroid $\LiealgebroidoverG$.
In particular, we have a map
\[ \groupoiddifferential{}:C^\infty(\ggroup)\to
\sections{\wedge^{k-1}\LiealgebroidoverG}\cong
C^\infty(\ggroup,\wedge^{k-1}\thetaalgebra) .\]

Since $\groupoiddifferential{}$ is a derivation, i.e.
$\groupoiddifferential{}(f_1f_2)=f_2\groupoiddifferential{}(f_1)
+f_1\groupoiddifferential{}(f_2)$, for all $ f_1,f_2\in
\cinf{\ggroup}$, $\groupoiddifferential{}$
induces a  $\wedge^{k-1} \thetaalgebra$-valued
vector field on $G$, which in turn can be identified with
a $\galgebra\otimes (\wedge^{k-1}\thetaalgebra)$-valued function on $G$.
Here we identify the tangent bundle $TG$ with $G\times \galgebra$ by
left translations. By skew symmetrization,  we thus
obtain a   $\galgebra\wedge(\wedge^{k-1}\thetaalgebra)$-valued function on $G$,
denoted by $ \wedgedelta$. More explicitly, we have
\begin{equation}\label{Eqt:groupoiddifferentiononf}
\groupoiddifferential{}(f)|_g=\minuspower{k-1}\inserts_{(L^*_{ {g}}\mathrm{d}f)}\wedgedelta|_g,
\qquad\forall~~ f\in \cinf{\ggroup},~g\in\ggroup.
\end{equation}

For $k\geq 1$, let
\begin{equation*}
W_k =\set{w\in\galgebra\wedge(\wedge^{k-1}\thetaalgebra)
\text{ s.t. } \inserts_{\zeta_1}\inserts_{\phi^*\zeta_2}w
=- \inserts_{\zeta_2}\inserts_{\phi^*\zeta_1}w,  \ \forall~ \zeta_1, \zeta_2\in
\galgebrastar}.
\end{equation*}

We have
\begin{lemma}\label{prop:wedgedelta1cocycle}
The function $\wedgedelta:  G\to \galgebra\wedge(\wedge^{k-1}\thetaalgebra)$
is a group $1$-cocycle valued in $W_k$.
Here $\ggroup$ acts on $\galgebra$ by the  adjoint action, and
on $\thetaalgebra$ by the induced  action from the crossed module structure.
\end{lemma}
\begin{proof}
Proposition  \ref{Thm:SigmaalongMdecomposition} describes how a multiplicative
vector field on a Lie groupoid looks like along the base manifold.
 Now we apply this theorem to the groupoid
$\crossedmoduletwogroup\toto G$. Identify $TG$ with $G\times \galgebra$
via left translations.
We write $\gthetaV|_g$ for the value of $\gthetaV$ at $(g, 1_{\thetagroup})$.
For all $g\in G$, we have
\begin{multline}\label{Eqn:Vatbasepowerseries}
\gthetaV|_g = L_{g*}\left(\frac{\Id-e^{-\phipush}}{\phipush}\big(\wedgedelta|_g\big)\right) \\
= L_{g*}\left(\wedgedelta|_g - \thalf \phipush \wedgedelta|_g +\frac{1}{3!}\phipush^2 \wedgedelta|_g+\cdots
+\tfrac{\minuspower{k-1}}{k!}\phipush^{k-1}\wedgedelta|_g\right)
.\end{multline}

Here $\phipush:  \ \wedge^\bullet(\crossedmodulealgebra)\to
 \wedge^\bullet(\crossedmodulealgebra)$
is a degree-$0$   derivation of the exterior algebra
$\wedge^\bullet(\crossedmodulealgebra)$ such that
  $\phipush( x +u)=\phi(u)$,  $\forall~  x \in\galgebra$, $u\in\thetaalgebra$,
  and,   by abuse of notation,  $L_{g*}$ denotes the tangent map of
 the left  translation by  $(g, 1_{\thetagroup})$ on the group
 $\crossedmoduletwogroup$.

Since $\gthetaV$ is multiplicative with respect to the group structure
on $\crossedmoduletwogroup$,   it follows that
\[ \gthetaV|_{gh}=L_{g*}\gthetaV|_{h}+R_{h*}\gthetaV|_{g} ,\]
where $R_{h*}$ denotes the tangent map of the right translation by $(h,1_{\thetagroup})$
in the group $\crossedmoduletwogroup$.
Substituting Eq.~\eqref{Eqn:Vatbasepowerseries} into the
equation above, we see that $\wedgedelta$ is indeed a Lie group 1-cocycle.

Moreover, Proposition~\ref{Thm:SigmaalongMdecomposition} implies that
\begin{equation}\label{symmetricityofwedgedelta}
\inserts_{\zeta_1}\inserts_{\phi^*\zeta_2}(\wedgedelta|_g)=-\inserts_{\zeta_2}\inserts_{\phi^*\zeta_1}(\wedgedelta|_g),
\quad\forall\zeta_1,\zeta_2\in \galgebrastar.
\end{equation}
As a consequence, $\wedgedelta$ takes values in $W_k$. This concludes the proof.
\end{proof}

Taking the derivative of $\wedgedelta$ at the unit:
\begin{equation}\label{eq:delta}
\skewdelta(x)= -\left.\frac{d}{dt}\right|_{t=0} \wedgedelta|_{\exp tx}, \quad\forall x\in\galgebra,
\end{equation}
we obtain the following
\begin{corollary}
Any multiplicative $k$-vector field on a Lie 2-group $\crossedmoduletwogroup$ induces a Lie
 algebra $1$-cocycle $\skewdelta:\galgebra\to\galgebra\wedge(\wedge^{k-1}\thetaalgebra)$.
\end{corollary}

\begin{lemma}\label{Thm:tangenttoTheata}
Identify  $\thetagroup$ with the   Lie subgroup $\{\unit_{\ggroup}\}\times \thetagroup$
 of $\crossedmoduletwogroup$, where $\unit_{\ggroup}$ is the unit element of $\ggroup$.
Then any multiplicative $k$-vector field $\gthetaV$ ($k\geq 1$) is
tangent to $\thetagroup$, and therefore  defines a multiplicative
$k$-vector field $\gthetaV|_{\thetagroup}$ on $\thetagroup$.
\end{lemma}

\begin{proof}
Let $i$ denote the inverse map of the groupoid $\crossedmoduletwogroup\toto G$,
as described in the proof of Proposition~\ref{thm:2groupstructuredetails}.
I.e.\ $i{(g,\alpha)}=(g\Phi(\alpha),\inverse{\alpha})$.
It is clear that $i_*\gthetaV=\minuspower{k+1}\gthetaV$ since $\gthetaV$ is multiplicative.
To prove the lemma, it suffices to prove that, for any function $f\in\cinf{\ggroup}$,
$\ba{\gthetaV}{\source^* f}|_{\thetagroup}=0$.
For all $\alpha\in\thetagroup\subset\crossedmoduletwogroup$, we have
\begin{multline*}
i_*\bigl(\ba{\gthetaV}{\source^* f}|_{\alpha}\bigr)=\minuspower{k+1}\ba{\gthetaV}{\target^* f}|_{i(\alpha)}
=\minuspower{k+1}\bigl( \gpoidleftmove{\groupoiddifferential{}(f)}\bigr)|_{i(\alpha)}\\
=\minuspower{k+1}L_{i(\alpha)*}\bigl({\groupoiddifferential{}(f)}|_{\target\circ i(\alpha)}\bigr)
=\minuspower{k+1}L_{i(\alpha)*}\bigl({\groupoiddifferential{}(f)}|_{  \unit_{\ggroup}}\bigr)
=0.\end{multline*}
 Here $\unit_{\ggroup}$ is the unit element of $\ggroup$, and
$L$ stands for the left translations with respect to the
groupoid structure.
The fact that ${\groupoiddifferential{}(f)}|_{\unit_{\ggroup}}=0$ is due to Eq.~\eqref{Eqt:groupoiddifferentiononf}
and Lemma~\ref{prop:wedgedelta1cocycle}.
\end{proof}

As an immediate consequence, the infinitesimal of $\gthetaV|_{\thetagroup}$ gives rise to a Lie algebra 1-cocycle
\begin{equation}\label{eq:omega}
\domega:\thetaalgebra\to\wedge^k \thetaalgebra
.\end{equation}
The pair $(\delta,\omega)$ as defined in Eqs.~\eqref{eq:delta} and~\eqref{eq:omega}
constitutes the \emph{infinitesimal data of $\gthetaV$}.

\subsubsection{Compatibility conditions}

This section is devoted to exploring the compatibility condition between
the infinitesimal data $\omega$ and $\delta$.
The main theorem is the following:

\begin{theorem}\label{Thm:Sigmatopair}
Let $\groupcrossedmoduletriple{\thetagroup}{\Phi}{\ggroup}$ be a crossed module of Lie groups.
A multiplicative $k$-vector field $\gthetaV$ on the Lie 2-group $\crossedmoduletwogroup$ associated to this
crossed module determines a pair of linear maps
\begin{gather*}
\domega:\thetaalgebra\to\wedge^k\thetaalgebra, \\
\skewdelta:\galgebra\to\galgebra\wedge(\wedge^{k-1}\thetaalgebra)
\end{gather*}
which satisfy the following three properties:
\begin{description}
\item[ID1]
$\phipush\smalcirc\omega=\skewdelta \smalcirc \phi$, i.e.\ the diagram
\[ \xymatrix{
\thetaalgebra \ar[r]^{\phi} \ar[d]^{\domega}
& \galgebra \ar[d]^{\skewdelta} \\
\wedge^k(\crossedmodulealgebra) \ar[r]_{\phipush}
& \wedge^k(\crossedmodulealgebra)
} \]
commutes;
\item[ID2]
$\skewdelta$ is a Lie algebra $1$-cocycle
 valued in $W_k$;
\item[ID3]
for all $x\in\galgebra$ and $u\in\thetaalgebra$,
\[ x \moduleaction\big(\omega(u)\big)-\omega( x \moduleaction
u)=\pr_{\wedge^k\thetaalgebra}\big(\ba{u}{\skewdelta(x)}\big) ,\]
where the bracket is taken in $\crossedmodulealgebra$.
\end{description}
\end{theorem}

First of all, we prove that  the $k$-differentials of $\gthetaV$
with respect to both the groupoid and the group structures can be expressed
in terms of the infinitesimal data $(\omega,\delta)$.
Since $\LiealgebroidoverG\cong G\times\theta$,
$\groupoiddifferential{}$ is completely determined
by two $\reals$-linear operators:
$\groupoiddifferential{}:C^\infty(G)\to C^\infty(G,\wedge^{k-1}\theta)$
and $\groupoiddifferential{}:C^\infty(G,\theta)\to C^\infty(G,\wedge^k\theta)$.
The latter is determined by its value on constant functions due to the Leibniz rule.

Let $\groupdifferential{}:\crossedmodulealgebra\to\wedge^{k}(\crossedmodulealgebra)$
be the $k$-differential with respect to the group structure on $\crossedmoduletwogroup$.

\begin{proposition}\label{prop:domegadefined}
The map $\domega:\theta\to\wedge^k \theta$ satisfies
\begin{equation}\label{Eqt:domegauanddifferentials}
\domega(u)=\groupdifferential{}(u)=\groupoiddifferential{}(u) ,\quad\forall u\in\thetaalgebra
.\end{equation}
\end{proposition}

\begin{proof}
Every $u\in\thetaalgebra\subset\crossedmodulealgebra$
(considered as a constant section of the Lie algebroid $\LiealgebroidoverG \cong G\times \theta$)
determines two vector fields on $\crossedmoduletwogroup$:
a vector field $\gpleftmove{u}$ invariant under left translations relatively to the group structure
and a vector field $\gpoidleftmove{u}$ invariant under left translations relatively to the groupoid structure $\crossedmoduletwogroup \toto G$.
It is simple to see that
\begin{equation}
\label{eq:gp-gpd}
\gpleftmove{u}=\gpoidleftmove{u}.
\end{equation}
Therefore,
\begin{equation*}
\gpoidleftmove{\groupdifferential{}{(u)}}=
\gpleftmove{\groupdifferential{}{(u)}}=\ba{\gthetaV}{\gpleftmove{u}}=
\ba{\gthetaV}{\gpoidleftmove{u}}=\gpoidleftmove{\groupoiddifferential{}{(u)}}.
\end{equation*}
By definition, we have
$\domega(u)=\groupdifferential{}(u)$. The conclusion follows.
\end{proof}

\begin{proposition}
\label{prop:groupdifferentialexplicit}
The $k$-differential
$\groupdifferential{}:\crossedmodulealgebra\to\wedge^k(\crossedmodulealgebra)$ satisfies
\begin{gather*}
\groupdifferential{}(u)=\domega(u), \quad \forall u\in\thetaalgebra, \\
\groupdifferential{}(x)=\frac{\Id-e^{-\phipush}}{\phipush}\big(\skewdelta(x)\big), \quad\forall x\in\galgebra
.\end{gather*}
\end{proposition}

\begin{proof}
It remains to prove the second equation, which follows from
a direct verification by applying Lemma~\ref{Lem:partialSigmacalculate} (2)
and Eq.~\eqref{Eqn:Vatbasepowerseries}.
\end{proof}

\begin{proof}[Proof of Theorem~\ref{Thm:Sigmatopair}]
According to Lemma~\ref{prop:wedgedelta1cocycle}, \textbf{ID2} holds.
It suffices to prove \textbf{ID1} and \textbf{ID3}.

Consider the $k$-differential $\groupoiddifferential{}:
\Gamma ( \wedge^{\bullet} A)\to \Gamma (\wedge^{\bullet+k-1} A)$
induced by $\gthetaV$.
For any  $u\in \thetaalgebra$,
$f\in C^\infty({\ggroup})$, we have
\begin{equation}\label{temp1}
\groupoiddifferential{}\ba{u}{f}=
\ba{\groupoiddifferential{}(u)}{f} +
\ba{u}{\groupoiddifferential{}(f)}
.\end{equation}

Next we evaluate both sides of Eq.~\eqref{temp1} at $e=\unit_{\ggroup}$.
Since $\wedgedelta$ is a Lie group $1$-cocycle according to Lemma~\ref{prop:wedgedelta1cocycle},
we have $\wedgedelta|_{e}=0$.
It thus follows from  Eq.~\eqref{Eqt:groupoiddifferentiononf} that $\groupoiddifferential{}\ba{u}{f}|_e=0$.
On the other hand, we have
\begin{multline*}
\ba{\groupoiddifferential{}(u)}{f}|_e=\ba{\domega(u)}{f}|_e
=\minuspower{k-1}\inserts_{\phi^*(\mathrm{d}{f})}\domega(u) \\
=\minuspower{k-1}\inserts_{ \mathrm{d}{f}}\big((\phipush\smalcirc\domega)(u)\big)
.\end{multline*}

Here we have used Proposition~\ref{prop:domegadefined} and the
equality $\inserts_{\phi^*\xi}=\inserts_{\xi}\circ \phipush$
in $\Hom(\wedge^k \theta,\wedge^{k-1}\theta)$ valid
for all $\xi\in\galgebrastar$.
Moreover, Eq.~\eqref{Eqt:groupoiddifferentiononf} implies that
$\groupoiddifferential{}(f)|_{e }=0$. Since the Lie algebroid $\LiealgebroidoverG$ is the
 transformation Lie algebroid $G\rtimes \theta$, we have
\begin{align*}
\ba{u}{\groupoiddifferential{}(f)}|_e
&=\left.\frac{d}{dt}\right|_{t=0}\bigl(\groupoiddifferential{}(f)|_{\exp t\phi(u)}\bigr) \\
&= \minuspower{k-1}\left.\frac{d}{dt}\right|_{t=0}\Bigl (\bigl(L^*_{ {\exp} t\phi(u)}{(\mathrm{d}f)}\bigr)\contract
{\wedgedelta|_{\exp t\phi(u)}}\Bigr) \\
&=-\minuspower{k-1}\inserts_{ \mathrm{d}{f}}\skewdelta(\phi(u)))
.\end{align*}

Here $\groupoiddifferential{}(f)$ is considered as a $(\wedge^{k-1}\theta)$-valued function on $G$, and the first equality follows from the Leibniz rule of the Lie algebroid axiom and the identity $\groupoiddifferential{}(f)|_{e }=0$.
Hence \textbf{ID1} follows immediately from Eq.~\eqref{temp1}.

On the other hand, the $k$-differential $\groupdifferential{}$ satisfies
\[ \groupdifferential{}\ba{x}{u}=\ba{\groupdifferential{}(x)}{u}
+\ba{x}{\groupdifferential{}(u)}, \quad \forall x\in\galgebra,u\in \thetaalgebra, \]
where the brackets stand for the Lie algebra bracket on $\crossedmodulealgebra$.
Applying Proposition~\ref{prop:groupdifferentialexplicit}, and comparing the
$\wedge^k\thetaalgebra$-terms of both sides of the equation above,
\textbf{ID3} follows immediately.
\end{proof}

\begin{proposition}\label{Prop:domega1cocycle}
The map $\domega:\theta\to\wedge^k\theta$ is a Lie algebra $1$-cocycle, i.e.\
\begin{equation*}
\domega\ba{u}{v}=\ba{\domega(u)}{v}+\ba{u}{\domega(v)}, \quad\forall u,v\in\thetaalgebra
.\end{equation*}
\end{proposition}

\begin{proof}
Using \textbf{ID1} and \textbf{ID3} from Theorem~\ref{Thm:Sigmatopair}, we have
\begin{align*}
\domega\ba{u}{v} &= \domega(\phi(u)\moduleaction v) \\
&= \phi(u)\moduleaction\domega(v)-\pr_{\wedge^k\thetaalgebra}(\ba{v}{\skewdelta(\phi(u))}) \\
&= \ba{u}{\domega(v)}-\pr_{\wedge^k\thetaalgebra}(\ba{v}{\phipush(\omega(u))}) \\
&= \ba{u}{\domega(v)}-\ba{v}{\domega(u)}
.\end{align*}
Here, in the last equality, we have used the identity
\[ \pr_{\wedge^k\thetaalgebra}(\ba{v}{\phipush(\zeta )})=
\ba{v}{\zeta }, \quad\forall\zeta\in\wedge^k\theta, \]
which can be verified by a straightforward computation.
\end{proof}

Now we extend the two maps $\domega$ and $\skewdelta$
to degree-$(k-1)$ derivations (which we denote by the same symbols by abuse of notation)
on the exterior algebra $\wedge^\bullet(\crossedmodulealgebra)$
by setting $\domega(\galgebra)=0$ and $\skewdelta(\thetaalgebra)=0$.

\begin{proposition}\label{Prop:bracketofmultiplicative}
Assume that $\gthetaV_1$ and  $\gthetaV_2$  are
multiplicative $k_1$- and $k_2$-vector fields on the  Lie
2-group $\crossedmoduletwogroup$.  Let $(\domega_1,\skewdelta_1)$
and $(\domega_2,\skewdelta_2)$ be their corresponding infinitesimals.
Then, the infinitesimal  $(\domega_3,\skewdelta_3)$  of
$\gthetaV_3=\ba{\gthetaV_1}{\gthetaV_2}$ is given by
 the following formulae:
\begin{align}\label{align:temp1}
&\domega_3={\domega_1}\circ{\domega_2}
-\minuspower{(k_1-1)(k_2-1)}{\domega_2}\circ{\domega_1},
\\\label{align:temp2}
&\skewdelta_3={(\skewdelta_1+\domega_1)}\circ{\skewdelta_2}
-\minuspower{(k_1-1)(k_2-1)}{(\skewdelta_2+\domega_2)}\circ{\skewdelta_1}.
\end{align}
\end{proposition}

\begin{proof}
Note that
\begin{multline*}
\groupdifferential{\gthetaV_3}=\groupdifferential{\ba{\gthetaV_1}{\gthetaV_2}}=\gradedcommutator{\groupdifferential{\gthetaV_1}}{\groupdifferential{\gthetaV_2}}
\\
={\groupdifferential{\gthetaV_1}}\circ{\groupdifferential{\gthetaV_2}}
-\minuspower{(k_1-1)(k_2-1)}{\groupdifferential{\gthetaV_2}}\circ{\groupdifferential{\gthetaV_1}}
.\end{multline*}
Hence Eqs.~\eqref{align:temp1} and~\eqref{align:temp2} follow immediately
from Proposition~\ref{prop:groupdifferentialexplicit}.
\end{proof}

By $\cala_k$ ($k\geq 1$) we denote the space of pairs $(\omega,\delta)$
of linear maps $\omega:\theta\to\wedge^k\theta$ and
$\delta:\galgebra\to\galgebra\wedge(\wedge^{k-1}\theta)$ satisfying
the three properties \textbf{ID1}, \textbf{ID2}, and \textbf{ID3} listed in Theorem~\ref{Thm:Sigmatopair}.
By $\cala_0$, we denote the space of all pairs $(\omega,\delta)$, where $\delta$ is the trivial map $\galgebra\to 0$
and $\omega:\theta\to\reals$ satisfies $\omega(x\moduleaction u)=0$,  for all $x\in\galgebra$ and $u\in\thetaalgebra$.

\begin{corollary}
When endowed with the bracket defined by Eqs.~\eqref{align:temp1} and~\eqref{align:temp2},
the direct sum $\bigoplus_{k\geq 0}\cala_k$ is a graded Lie algebra.
\end{corollary}

\subsection{The universal lifting theorem}

\subsubsection{Statement of the main theorems}

The converse of Theorem \ref{Thm:Sigmatopair} holds as well.
\begin{theorem}\label{Thm:pairtoSigma}
Let $\crossedmoduletwogroup$ be a Lie 2-group, where both
$\ggroup$ and $\thetagroup$ are connected and simply connected.
Given any $(\omega,\delta)\in\cala_k$, there exists a unique multiplicative $k$-vector field
$\gthetaV$ on $\crossedmoduletwogroup$, whose infinitesimal corresponds to $(\domega,\skewdelta)$.
\end{theorem}

An immediate consequence is the following main result of the paper:

\begin{theorem}[universal lifting theorem]
Let $\crossedmoduletwogroup$ be a Lie 2-group, where both
$\ggroup$ and $\thetagroup$ are connected and simply connected
Lie groups with Lie algebras $\galgebra$ and $\theta$, respectively.
There is a canonical isomorphism of graded Lie algebras
\[ \bigoplus_{k\geq 0} \strictmultiplicative{k}{\crossedmoduletwogroup}\cong
\bigoplus_{k\geq 0} \cala_k .\]
\end{theorem}

\subsubsection{From infinitesimal data to $k$-differentials of
the Lie algebra  $\crossedmodulealgebra$}

The $k=0$ case is obvious, so we will assume $k\geq 1$ below.
We will divide the proof of Theorem~\ref{Thm:pairtoSigma} into several steps.

First, since $\skewdelta$ is a Lie algebra $1$-cocycle, it integrates into a Lie group $1$-cocycle
\[ \wedgedelta:\ggroup\to\galgebra\wedge(\wedge^{k-1}\thetaalgebra) \] such that
\begin{equation}\label{eq:delta2}
\skewdelta(x)=-\left.\frac{d}{dt}\right|_{t=0}\wedgedelta|_{\exp tx}, \quad\forall x\in\galgebra
.\end{equation}
Clearly the map $\wedgedelta$ takes values in $W_k$.

Let $\wedgeomega:\thetagroup\to\wedge^k\thetaalgebra$ be the group $1$-cocycle  integrating $\domega$.
As a direct consequence of property \textbf{ID1} from Theorem~\ref{Thm:Sigmatopair},
we have
\begin{equation}\label{Eqn:wedgedeltaomegacomptible1}
\wedgedelta|_{\Phi(\alpha)}=\phipush(\wedgeomega|_{\alpha}), \quad\forall g\in\ggroup, \alpha\in\thetagroup
.\end{equation}

Define a linear map $\partial:\crossedmodulealgebra\to\wedge^{k}(\crossedmodulealgebra)$ by
\begin{equation}\label{Eqt:defnpartialoutoftwodata}
\left\{ \begin{aligned}
\partial (u) &=\domega (u), &&\forall u\in\thetaalgebra; \\
\partial (x)  &=\frac{\Id-e^{-\phipush}}{\phipush}~\skewdelta (x), &&\forall x\in\galgebra
.\end{aligned}
\right.
\end{equation}

\begin{proposition}
\label{Prop:partial1cocycle} The operator $\partial$ defines
 a Lie algebra $k$-differential for the Lie algebra $\crossedmodulealgebra$.
\end{proposition}
\begin{proof}

It suffices to prove that $\partial$ is a Lie algebra $1$-cocycle for the Lie algebra $\crossedmodulealgebra$.
In fact, Proposition~\ref{Prop:domega1cocycle} implies that
\[ \partial\ba{u}{v}=\ba{\partial u}{v}+\ba{u}{\partial v}, \quad\forall u,v\in\thetaalgebra .\]

On the other hand, it follows from a direct verification that
\[ \phipush\ba{x}{w}=\ba{x}{\phipush(w)}, \quad\forall x\in\galgebra, w\in\wedge^\bullet(\crossedmodulealgebra) .\]
As a consequence, applying the operator $\frac{\Id-e^{-\phipush}}{\phipush}$ to both sides of the equation:
\[ \skewdelta\ba{x}{y}=\ba{\skewdelta x}{y}+\ba{x}{\skewdelta y} ,\]
we obtain
\[ \partial\ba{x}{y}=\ba{\partial x}{y}+\ba{x}{\partial y}, \quad\forall x,y\in\galgebra .\]

It remains to prove the identity
\[ \partial\ba{x}{u}=\ba{\partial x}{u}+\ba{x}{\partial u}, \quad\forall x\in\galgebra,u\in \thetaalgebra .\]

Since $\partial\ba{x}{u}-\ba{x}{\partial u}=\pr_{\wedge^k\thetaalgebra}(\ba{\skewdelta(x)}{u})$
according to property \textbf{ID3} from Theorem~\ref{Thm:Sigmatopair}, it suffices to prove that
\begin{equation*}
\ba{\partial x}{u}=\pr_{\wedge^k\thetaalgebra}(\ba{\skewdelta(x)}{u})
.\end{equation*}
Now
\begin{align*}
\ba{\partial x}{u}
&=\left[\sum_{i=0}^{k-1}\frac{\minuspower{i}}{(i+1)!}\phipush^i(\skewdelta(x)) , u\right] \\
&=\sum_{j=0}^{k-1} \pr_{\wedge^j\galgebra\wedge(\wedge^{k-j}\thetaalgebra)}
\left[ \sum_{i=0}^{k-1}\frac{\minuspower{i}}{(i+1)!}\phipush^i(\skewdelta(x)) , u\right]
.\end{align*}
Using the definitions of $\skewdelta$ and $D_{\phi}$, we obtain the following
identity:
\[ \pr_{\wedge^{k }\thetaalgebra}
\left[ \sum_{i=0}^{k-1}\frac{\minuspower{i}}{(i+1)!}\phipush^i(\skewdelta(x)) , u \right]
=\pr_{\wedge^k\thetaalgebra}\ba{\skewdelta(x)}{u}
.\]
For $1\leq j\leq k-1$, the sum
$\pr_{\wedge^j\galgebra\wedge(\wedge^{k-j}\thetaalgebra)}
\left[ \sum_{i=0}^{k-1}\frac{\minuspower{i}}{(i+1)!}\phipush^i(\skewdelta(x)) , u \right]$
contains only the two terms
\[ \pr_{\wedge^j\galgebra\wedge(\wedge^{k-j}\thetaalgebra)}
\left( \left[ \frac{\minuspower{j-1}}{j!}\phipush^{j-1}(\skewdelta(x)) , u \right]
+ \left[ \frac{\minuspower{j}}{(j+1)!}\phipush^j(\skewdelta(x)) , u \right] \right) \]
and thus reduces to
\begin{equation}\label{Eqt:templine}
\frac{\minuspower{j-1}}{j!}\pr_{\wedge^j\galgebra\wedge(\wedge^{k-j}\thetaalgebra)}
\Big( \left[ \phipush^{j-1}(\skewdelta(x)) , u\right] -\frac{1}{j+1}
\left[ \phipush^j(\skewdelta(x)) , u \right] \Big) .\end{equation}
To prove that it vanishes, we need a couple of lemmas.

\begin{lemma}\label{Lem:projectionbracketzero}
For any $v\in\wedge^{k-1}\thetaalgebra$ and $1\leq l\leq k-1$, we have
\[ \pr_{\wedge^{l-1}\galgebra\wedge(\wedge^{k-l}\thetaalgebra)}
\Big( \left[ l \phipush^{l-1}(v)- \phipush^{l}(v) , u \right] \Bigl)=0,
\quad\forall u\in\thetaalgebra .\]
\end{lemma}
\begin{proof} This follows from a straightforward computation,
which is left to the reader.
\end{proof}

From Lemma~\ref{Lem:skewdeltaderive1}, it follows that,
for any $w\in W_k$ and $j\geq 1$, we have
\begin{equation}\label{Lem:skewdeltaderive}
\inserts_{\phi^*\zeta}(\phipush^{j-1}w)=\phipush^j(\inserts_{\zeta}w)
=\frac{1}{j+1}\inserts_{\zeta}(\phipush^j w), \quad\forall\zeta\in\galgebrastar
.\end{equation}

Now we return to the proof of Proposition~\ref{Prop:partial1cocycle}.
It remains to prove that~\eqref{Eqt:templine} vanishes.
Indeed, for any $\zeta\in\galgebrastar$, we have
\begin{multline*}
\inserts_{\zeta}\pr_{\wedge^j\galgebra\wedge(\wedge^{k-j}\thetaalgebra)}
\Bigr( \left[ \phipush^{j-1}(\skewdelta(x)) , u \right]
-\frac{1}{j+1} \left[ \phipush^j(\skewdelta(x)) , u \right] \Bigl) \\
=\pr_{\wedge^{j-1}\galgebra\wedge(\wedge^{k-j}\thetaalgebra)}
\Bigr( \left[ \inserts_{\zeta}\phipush^{j-1}(\skewdelta(x)) , u \right]
-\frac{1}{j+1} \left[ \inserts_{\zeta}\phipush^j(\skewdelta(x)) , u \right] \Bigl) \\
=\pr_{\wedge^{j-1}\galgebra\wedge(\wedge^{k-j}\thetaalgebra)}
\Bigr( \left[ j \phipush^{j-1}(\inserts_{\zeta}\skewdelta(x))-
\phipush^{j}(\inserts_{\zeta}\skewdelta(x)) , u \right] \Bigl)=0.
\end{multline*}
Here in the last two steps, we have used  Eq.~\eqref{Lem:skewdeltaderive}
and Lemma~\ref{Lem:projectionbracketzero}.
This concludes the proof of the proposition.
\end{proof}

\subsubsection{Multiplicative with respect to the  groupoid structure}

As a consequence of Proposition \ref{Prop:partial1cocycle}, we obtain
a $k$-vector field
$\Sigma$ on $\crossedmoduletwogroup$, which is multiplicative
with respect to the group structure
and whose induced $k$-differential with respect to the group structure
on $\crossedmoduletwogroup$ is
$\partial$. Now we need to prove that $\Sigma$ is also
multiplicative with respect to the groupoid structure on
 $\crossedmoduletwogroup\toto G$. For this purpose, we need
an explicit expression of $\Sigma$.
Since $\Sigma$ is multiplicative with respect to the group structure
on $\crossedmoduletwogroup$,
it suffices to find an explicit expression of $\Sigma$
along the subgroups $\{\unit_G\}\times \thetagroup$ and
$G \times \{\unit_{\thetagroup}\}$, respectively.
The next two lemmas are devoted to this investigation.

The following lemma is immediate.
\begin{lemma}\label{Prop:SigmaatTheta}
Identify $\thetagroup$ with the subgroup $\{\unit_G\}\times\thetagroup$ of $\crossedmoduletwogroup$.
Then $\Sigma$ is tangent to $\thetagroup$  and therefore induces a multiplicative
$k$-vector field $\Sigma|_{\thetagroup}$ on $\thetagroup$.
Moreover, $\Sigma|_{\alpha}=L_{\alpha*}(\wedgeomega|_{\alpha})$, for all $\alpha\in\thetagroup$.
\end{lemma}

Next, we have
\begin{lemma}\label{Prop:SigmaatG}
Along the Lie subgroup $\ggroup\cong\ggroup\times\{\unit_\Theta\}\subset G\times\Theta$,
$\Sigma$ can be explicitly expressed by the following formula:
\begin{equation}\label{Eqn:Sigmaatbasepowerseries2}
\Sigma|_{g}=L_{g*}\left(\frac{\Id-e^{-\phipush}}{\phipush}\big(\wedgedelta|_g\big)\right),
\quad\forall g\in\ggroup
.\end{equation}
Moreover, for any $\zeta\in\galgebrastar$,  we have
\begin{equation}\label{Eqt:insertszetaplusphiUpseta}
\inserts_{(L^*_{\inverse{g}}\zeta+\phi^*\zeta)}(\Sigma|_g)
=\inserts_{\zeta}(\wedgedelta|_g)\,.
\end{equation}
Here $L^*_{\inverse{g}}\zeta\in T_g^*\ggroup$ and $\phi^*\zeta\in\thetaalgebrastar=T^*_{\unit_{\thetagroup}}\thetagroup$.
\end{lemma}

\begin{proof}
Eq.~\eqref{Eqn:Sigmaatbasepowerseries2} follows from integrating $\partial(x)$ in Eq.~\eqref{Eqt:defnpartialoutoftwodata}.
To prove Eq.~\eqref{Eqt:insertszetaplusphiUpseta}, according to Eq.~\eqref{Lem:skewdeltaderive}, we have
\[ \inserts_{\phi^*\zeta}\left(\tfrac{\minuspower{j-1}}{j!}\phipush^{j-1}\big(\wedgedelta|_g\big)\right)
=\tfrac{\minuspower{j-1}}{j!}\tfrac{1}{j+1}\inserts_{\zeta}\phipush^j\big(\wedgedelta|_g\big)
=- \inserts_{\zeta}\left(\tfrac{\minuspower{j}}{(j+1)!}\phipush^{j}\big(\wedgedelta|_g\big)\right) .\]
The conclusion thus follows immediately by using Eq.~\eqref{Eqn:Sigmaatbasepowerseries2}.
\end{proof}

\begin{proposition}
The $k$-vector field $\Sigma$ is also multiplicative with respect to
the groupoid structure on $\crossedmoduletwogroup\toto\ggroup$.
\end{proposition}

\begin{proof}
We divide the proof into three steps.
\newline
\textbf{(1)} \textit{The base manifold $G$ is coisotropic with respect to $\Sigma$.}
\newline
For every $g\in\ggroup$, we have
$T_{(g,\unit_\Theta)}(G\ltimes\Theta)\cong T_g\ggroup\oplus\thetaalgebra$.
The conormal space of $T_g\ggroup$ can thus be canonically identified with $\thetaalgebrastar$.
It follows that $G$ is coisotropic with respect to $\Sigma$ since $\Sigma|_g$
does not contain any $(\wedge^k\thetaalgebra)$-components according to
Lemma~\ref{Prop:SigmaatG}.
\newline
\textbf{(2)} \textit{For every $\xi\in\Omega^1(\ggroup)$, $\inserts_{\target^*(\xi)}\Sigma$
is left-invariant with respect to the groupoid structure.}
\newline
For every $(g,\alpha)\in\crossedmoduletwogroup$, we identify $T_{(g,\alpha)} (\crossedmoduletwogroup )$
with $T_g\ggroup\oplus T_{\alpha}\thetagroup$. Since $\wedgeomega|_{\alpha}$ takes values in
$\wedge^k\thetaalgebra$, we have
\begin{equation*}
\Sigma|_{(g,\alpha)}=\Sigma|_{ g\gpmulti \alpha
}=L_{g*}(\Sigma|_{\alpha})+R_{\alpha*}(\Sigma|_{g})
=  L_{g*}L_{\alpha*}(\wedgeomega|_{\alpha})+R_{\alpha*}(\Sigma|_{g})
.\end{equation*}
Let $m$ be the point $\target(g,\alpha)=g\Phi(\alpha)$ of $\ggroup$.
Choose a $\zeta\in\galgebrastar$ and set $\xi|_{m}=L^*_{\inverse{m}}\zeta\in T^*_{m}\ggroup$.
We have, for all $u\in\thetaalgebra$,
\begin{equation}\label{eqn:temptemp1}
\inserts_{\target^*(\xi)} L_{g*}L_{\alpha*}u=\inserts_{\phi^*(\zeta)}u,
\end{equation}
which follows from the identity
\begin{equation*}
(\target\circ L_g\circ L_{\alpha})(\unit_{\ggroup},\beta)=g\Phi(\alpha)\Phi(\beta)=
(L_{g\Phi(\alpha)}\circ\Phi )(\beta), \quad\forall\beta\in\thetagroup
.\end{equation*}
Also note that, for all $V\in T_{(g, \unit_\Theta)}(\crossedmoduletwogroup)$,
\begin{equation}\label{eqn:temptemp2}
\inserts_{\target^*(\xi)}R_{\alpha*}V=\inserts_{\big(L^*_{\inverse{g}}\coAdjoint{\Phi(\alpha)}\zeta\big)}V
+\inserts_{\big(\phi^*\coAdjoint{\Phi(\alpha)}\zeta\big)}V
.\end{equation}
To prove this identity, we observe that
\begin{equation*}
\big(\target\circ R_{\alpha}\circ L_{g}\big)(h,\beta)=gh\Phi(\beta)\Phi(\alpha),
\quad\forall h\in\ggroup,\beta\in\thetagroup
,\end{equation*}
which implies that
\[ \big(\target_*\circ R_{\alpha*}\circ L_{g*}\big)(x,u)
=\big(L_{g\Phi(\alpha)*}\circ\Adjoint{\Phi(\alpha^{-1})}\big)(x+\phi(u)),
\quad\forall x\in \galgebra,u\in\thetaalgebra .\]
Thus Eq.~\eqref{eqn:temptemp2} follows from a straightforward verification.

Applying Eq.~\eqref{eqn:temptemp1}, we obtain
\begin{align*}
\inserts_{\target^*(\xi)}L_{g*}(\Sigma|_{\alpha})
&= \inserts_{\target^*(\xi)}\big(L_{g*}L_{\alpha*}(\wedgeomega|_{\alpha})\big) && \\
&= L_{g*}L_{\alpha*}(\inserts_{\phi^*\zeta}\wedgeomega|_{\alpha}) && \\
&= L_{g*}L_{\alpha*}\big(\inserts_{ \zeta}\phipush\wedgeomega|_{\alpha}\big) &&
\text{(by Eq.~\eqref{Eqn:wedgedeltaomegacomptible1})} \\
&= L_{g*}L_{\alpha*}\big(\inserts_{\zeta}\wedgedelta|_{\Phi(\alpha)}\big) . &&
\end{align*}
Using Eq.~\eqref{eqn:temptemp2} and Lemma~\ref{Prop:SigmaatG}), we have
\begin{align*}
\inserts_{\target^*(\xi)} R_{\alpha*}(\Sigma|_{g})
&= R_{\alpha*}\Big(\inserts_{L^*_{\inverse{g}}(\coAdjoint{\Phi(\alpha)}\zeta)}
(\Sigma|_g)+\inserts_{\phi^*\coAdjoint{\Phi(\alpha)}\zeta}(\Sigma|_g)\Big) \\
&= L_{g*}R_{\alpha*}\bigl(\inserts_{\coAdjoint{\Phi(\alpha)}\zeta}~\wedgedelta|_g\bigr).
\end{align*}
Therefore we have
\begin{align*}
(\inserts_{\target^*(\xi)}\Sigma)|_{(g,\alpha)}
&= \inserts_{\target^*(\xi)}L_{g*}(\Sigma|_{\alpha})+\inserts_{\target^*(\xi)} R_{\alpha*}(\Sigma|_{g}) \\
&= L_{g*}L_{\alpha*}\big(\inserts_{\zeta}\wedgedelta|_{\Phi(\alpha)}\big)
+L_{g*}R_{\alpha*}\big(\inserts_{\coAdjoint{\Phi(\alpha)}\zeta}\wedgedelta|_g\big) \\
&= L_{g*}L_{\alpha*}\Big(\inserts_{\zeta}\wedgedelta|_{\Phi(\alpha)}
+\Adjoint{\inverse{\alpha}}(\inserts_{\coAdjoint{\Phi(\alpha)}\zeta}\wedgedelta|_g)\Big) \\
&= L_{g*}L_{\alpha*}\inserts_{\zeta}\Big(\wedgedelta|_{\Phi(\alpha)}
+(\Phi(\inverse{\alpha}))_*\wedgedelta|_g)\Big) \\
&= L_{g*}L_{\alpha*}\inserts_{\zeta}\wedgedelta|_{g\Phi(\alpha)}
,\end{align*}
where, in the last step, we used the fact that $\wedgedelta$ is a Lie group $1$-cocycle.
In particular, we have
\[ (\inserts_{\target^*(\xi)}\Sigma)|_{ g\Phi(\alpha)}
=(\inserts_{\target^*(\xi)}\Sigma)|_{(g\Phi(\alpha),\unit_{\thetagroup})}
=L_{g \Phi(\alpha)*}\inserts_{\zeta}\wedgedelta|_{g\Phi(\alpha)} \]
and therefore
\[ (\inserts_{\target^*(\xi)}\Sigma)|_{(g,\alpha)}
=L^{\mathrm{gpd}}_{(g,\alpha)*}(\inserts_{\target^*(\xi)}\Sigma)|_{g\Phi(\alpha)} .\]
This proves that $\inserts_{\target^*(\xi)}\Sigma$ is indeed left-invariant with respect to the groupoid structure.
\newline
\textbf{(3)} \textit{For every $X\in \sections{\LiealgebroidoverG}$, $\ba{\Sigma}{\gpoidleftmove{X}}$
is left-invariant with respect to the groupoid structure.}
\newline
It suffices to consider $X=fu$, where $f\in C^\infty({\ggroup})$ and $u\in\thetaalgebra$
being considered as a constant section of $\LiealgebroidoverG\cong G\times\theta$. Then,
\begin{align*}
\ba{\Sigma}{\gpoidleftmove{X}} &= \ba{\Sigma}{(\target^*f)\gpoidleftmove{u}} && \\
&= (\target^*f)\ba{\Sigma}{\gpoidleftmove{u}}+\ba{\Sigma}{\target^*f}\wedge\gpoidleftmove{u}
&& \text{(by Eq.~\eqref{eq:gp-gpd})} \\
&= (\target^*f)\ba{\Sigma}{\gpleftmove{u}}+\minuspower{k-1}
\inserts_{\target^*\mathrm{d}f}\Sigma\wedge \gpoidleftmove{u} && \\
&= (\target^*f)\gpleftmove{\groupdifferential{\Sigma}(u)}+\minuspower{k-1}
\inserts_{\target^*\mathrm{d}f}\Sigma\wedge \gpoidleftmove{u} && \text{(by Eq.~\eqref{eq:gp-gpd})} \\
&= (\target^*f)\gpoidleftmove{\groupdifferential{\Sigma}(u)}+\minuspower{k-1}
\inserts_{\target^*\mathrm{d}f}\Sigma\wedge \gpoidleftmove{u}, &&
\end{align*}
which is clearly left-invariant according to Claim~\textbf{(2)}.
\newline
Finally, Claims \textbf{(1)}, \textbf{(2)}, and \textbf{(3)} imply that $\Sigma$ is indeed multiplicative
by Lemma~\ref{Lem:Sigmagroupoidmultiplicativeiff}.
\end{proof}

\begin{proof}[Proof of Theorem~\ref{Thm:pairtoSigma}]
From the infinitesimal data $(\domega,\skewdelta)$, we have constructed
a multiplicative $k$-vector field $\Sigma$ on the 2-group $\crossedmoduletwogroup$.
Assume that the infinitesimal data corresponding to $\Sigma$ is $(\domega',\skewdelta')$.
Proposition~\ref{prop:groupdifferentialexplicit} implies that $\domega'$
and $\skewdelta'$ can be recovered from $\groupdifferential{}$,
the $k$-differential of $\Sigma$ with respect to the group structure on $\crossedmoduletwogroup$,
by the following relations:
\begin{equation*}
\left\{ \begin{aligned}
\groupdifferential{}(u) &= \domega'(u),\\
\groupdifferential{}(x) &= \frac{\Id-e^{-\phipush}}{\phipush} \big(\skewdelta'(x)\big)
.\end{aligned} \right.
\end{equation*}
Since $\partial$ is defined by Eq.~\eqref{Eqt:defnpartialoutoftwodata} and $\Sigma$ integrates $\partial$,
$\groupdifferential{}$ must coincide with $\partial$. Hence it follows that $\domega'=\domega$ and $\skewdelta'=\skewdelta$.

Since both $\ggroup$ and $\thetagroup$ are connected and simply connected, so must be $\crossedmoduletwogroup$.
Hence the multiplicative vector field $\Sigma$ that integrates $\partial$ must be unique.
\end{proof}

\section{Quasi-Poisson Lie 2-groups}

Throughout this section,
$\groupcrossedmoduletriple{\thetagroup}{\Phi}{ \ggroup}$ denotes a
Lie group crossed module,  and $\crossedmoduletwogroup$ its
associated  Lie 2-group. By
$\crossedmoduletriple{\thetaalgebra}{\phi}{\galgebra} $,  we denote  its
corresponding Lie algebra crossed module,
 and by $\crossedmodulealgebra$ the semidirect product Lie algebra.

\subsection{Quasi-Poisson  Lie 2-groups}

\begin{definition}\label{defn:qusi-Poisson}
A quasi-Poisson structure on a Lie 2-group $\crossedmoduletwogroup$
is a pair $(\gthetaPi,\inteta)$, where $\gthetaPi\in\strictmultiplicative{2}{\crossedmoduletwogroup}$
is a multiplicative bivector field, $\inteta:\ggroup\to\wedge^3\thetaalgebra$ is a Lie group $1$-cocycle
such that
\begin{gather}
\label{Eqn:halfPiPiCinteta}
\thalf\ba{\gthetaPi}{\gthetaPi}=\gpoidleftmove{\inteta}-\gpoidrightmove{\inteta} ,
\intertext{and}
\label{Eqn:Pietazero}
\ba{\gthetaPi}{\gpoidleftmove{\inteta}}=0 .
\end{gather}
Here $\inteta$ is considered as a section in $\sections{\wedge^3\LiealgebroidoverG}$.
When $\inteta$ is zero, $\gthetaPi$ defines a Poisson structure on $\crossedmoduletwogroup$.
In this case, we say that $(\crossedmoduletwogroup,\gthetaPi)$ is a Poisson 2-group.
\end{definition}

It is clear that $\crossedmoduletwogroup\toto G$ together with $(\gthetaPi, \inteta)$
is a quasi-Poisson groupoid~\cite{MR2911881}.

The main result of this  section is the following:
\begin{theorem}\label{Thm:quasi-lie2-bioneteonequasipoisson}
Any quasi-Poisson Lie 2-group $(\gthetaPi,\inteta)$ on $\crossedmoduletwogroup$
naturally induces a quasi-Lie 2-bialgebra.
\newline
Conversely, given a quasi-Lie 2-bialgebra $(\theta,\galgebra,t)$ as in
Definition~\ref{Defn:Lie2bialgebra}, if both $\ggroup$ and $\thetagroup$
are connected and simply-connected Lie groups with Lie algebras
$\theta$ and $\galgebra$, respectively, then $\crossedmoduletwogroup$
admits a quasi-Poisson Lie 2-group structure whose infinitesimal
is isomorphic to the given quasi-Lie 2-bialgebra.
\end{theorem}

The proof is deferred to Section~\ref{Section:proofofmain}.
In fact, from its proof, it is clear that exactly the same conclusion holds
between Poisson Lie 2-groups and Lie 2-bialgebras.
Thus, as an immediate consequence, we obtain the following analogue of
a classical theorem of Drinfeld in the context of 2-groups.

\begin{corollary}\label{Thm:quasiPoissonLie2group1to1Lie2bialgebral3=0}
\begin{enumerate}
\item There is a one-to-one correspondence between connected and simply connected
quasi-Poisson Lie 2-groups and quasi-Lie 2-bialgebras.
\item There is a one-to-one correspondence between connected and simply-connected
Poisson Lie 2-groups and Lie 2-bialgebras.
\end{enumerate}
\end{corollary}

\subsection{Multiplicative $k$-vector fields generated by group 1-cocycles}

\begin{lemma}
For any $u\in \thetaalgebra$, we have
\begin{gather}
\label{Eqt:Ad1minusphi-appendix}
\big(\Adjoint{\gpinverse{(h,\beta)}}\circ(\id_{\thetaalgebra}-\phi)\big)(u)
=\big((\id_{\thetaalgebra}-\phi)\circ\inverse{h}_*\big)(u) ,\\
\label{Eqn:Adjoint2-appendix}
\Adjoint{\gpinverse{(h,\beta)}}(u)=
\big(\Adjoint{\inverse{\beta}}\circ\inverse{h}_*\big)(u)
=\inverse{\big(h\Phi(\beta)\big)}_*(u) .
\end{gather}
Here $\id_{\thetaalgebra}$ denotes the identity map on $\thetaalgebra$.
\end{lemma}

\begin{proof}
A straightforward computation yields that, for all $(g,\alpha)\in\crossedmoduletwogroup$,
\begin{align*}
\Adjoint{\gpinverse{(h,\beta)}}(g,\alpha)
&= \gpinverse{(h,\beta)}\gpmulti(g,\alpha)\gpmulti(h,\beta) \\
&= \Bigl({\inverse{h}}gh,\bigl(({\inverse{h}}\inverse{g}h)\groupaction\inverse{\beta}\bigr)
(\inverse{h}\groupaction\alpha)\beta\Bigr) \\
&= \Bigl(\Adjoint{\inverse{h}}g,\bigl((\Adjoint{\inverse{h}}\inverse{g})
\groupaction\inverse{\beta}\bigr)(\inverse{h}\groupaction\alpha)\beta\Bigr)
.\end{align*}
In particular, we have
\begin{equation*}
\Adjoint{\gpinverse{(h,\beta)}}(\Phi(\inverse{\alpha}),\alpha)
=\big(\Phi(\inverse{h}\groupaction \inverse{\alpha}),\inverse{h}\groupaction\alpha)\big)
.\end{equation*}
Eq.~\eqref{Eqt:Ad1minusphi-appendix} thus follows immediately
by taking the tangent map at $\alpha=1_{\Theta}$.
Similarly, we have
\begin{equation*}
\Adjoint{\gpinverse{(h,\beta)}}(\unit_{\ggroup},\alpha)
=\big(\unit_{\ggroup},\inverse{(h\Phi(\beta))}\groupaction\alpha)\big).
\end{equation*}
Eq.~\eqref{Eqn:Adjoint2-appendix} follows by taking the tangent map at $\alpha=1_{\Theta}$.
\end{proof}

\begin{proposition}\label{Lem:multiplicativeoutofcocycle}
Let $\intlambda:\ggroup\to\wedge^l\thetaalgebra$
be a Lie group $1$-cocycle, and $\dlambda:\galgebra\to\wedge^l\thetaalgebra$
the corresponding Lie algebra $1$-cocycle.
\begin{enumerate}
\item\label{m1}
The $l$-vector field
\begin{equation*}
\gthetaC_{\intlambda}=\gpoidleftmove{\intlambda}-\gpoidrightmove{\intlambda}
\end{equation*}
on the 2-group $\crossedmoduletwogroup$ is multiplicative.
Here $\intlambda$ is considered as a section in $\sections{\wedge^l\LiealgebroidoverG}$,
and $\gpoidleftmove{\intlambda}$ and $\gpoidrightmove{\intlambda}$, respectively,
denote the left- and right-invariant $l$-vector fields on the groupoid $\crossedmoduletwogroup\toto G$.
\item\label{m2}
The infinitesimal data of $\gthetaC_{\intlambda}$ is
\begin{align*}
\domega_{\dlambda} &= \dlambda\circ\phi:\theta\to\wedge^l \theta; \\
\skewdelta_{\dlambda} &= \phipush\circ\dlambda:\galgebra\to\galgebra\wedge(\wedge^{l-1}\theta)
.\end{align*}
\item\label{m3}
Let $\groupoiddifferential{}:\sections{\wedge^{\bullet}\LiealgebroidoverG }
\to\sections{\wedge^{\bullet +k-1 }\LiealgebroidoverG }$ be the $k$-differential on the Lie algebroid $A$
induced by a multiplicative $k$-vector field $\gthetaV$ on $\crossedmoduletwogroup$.
Then, the section $\intsigma=\groupoiddifferential{}(\intlambda)\in\sections{\wedge^{k+l-1}\LiealgebroidoverG }$,
considered as a map $\ggroup\to\wedge^{k+l-1}\thetaalgebra$, is a Lie group $1$-cocycle.
The corresponding Lie algebra $1$-cocycle $\dsigma:\galgebra\to\wedge^{k+l-1}\thetaalgebra$ is
\[ \dsigma=\domega\circ \dlambda-\minuspower{(k-1)(l-1)}\dlambda\circ\ddelta ,\]
where $(\domega,\ddelta)$ is the infinitesimal of $\gthetaV$.
\end{enumerate}
\end{proposition}

\begin{proof}
\textbf{1)} It is clear that $\gthetaC_{\intlambda}$ is multiplicative
with respect to the groupoid structure. It suffices to show that $\gthetaC_{\intlambda}$
is also multiplicative with respect to the group structure on $G\ltimes\Theta$.
Define $c:G\ltimes\Theta\to\wedge^{l}(\crossedmodulealgebra)$ by
\[ c|_{(g,\alpha)}= L_{\gpinverse{(g,\alpha)}*}\big(\gthetaC_{\intlambda}|_{(g,\alpha)}\big) ,\]
where $L$ stands for the  group left translations.
It is well known that $\gthetaC_{\intlambda}$ is multiplicative
with respect to the group structure if and only if $c$ is a group $1$-cocycle, i.e.\
\begin{equation}\label{Eqt:c}
c|_{(g,\alpha)\gpmulti(h,\beta)}=c|_{(h,\beta)}+\Adjoint{\gpinverse{(h,\beta)}}(c|_{(g,\alpha)})
.\end{equation}
Now a direct calculation yields
\begin{equation}\label{Eq:candlambda}
c|_{(g,\alpha)}=\intlambda|_{ g\Phi(\alpha)}-(\id_{\thetaalgebra}-\phi)\intlambda|_{g}
.\end{equation}
Here $(\id_{\thetaalgebra}-\phi)$ extends naturally to a map
$\wedge^l\thetaalgebra\to\wedge^l\thetaalgebra$, i.e.\
\begin{multline*}
(\id_{\thetaalgebra}-\phi)(u_1\wedge u_2\wedge\cdots\wedge u_l)
=(\id_{\thetaalgebra}-\phi)u_1\wedge(\id_{\thetaalgebra}-\phi)u_2\wedge\cdots\wedge
(\id_{\thetaalgebra}-\phi)u_l
,\end{multline*}
for all $u_1,\cdots,u_l\in\thetaalgebra$.
Using Eq.~\eqref{Eq:candlambda} and the assumption that $\intlambda$ is a $1$-cocycle, we have
\begin{align*}
& \text{r.h.s\ of Eq.~\eqref{Eqt:c}} \\
=& \intlambda|_{h\Phi(\beta)}-(\id_{\thetaalgebra}-\phi)\intlambda|_{h}
+\inverse{(h\Phi(\beta))}_*(\intlambda|_{g\Phi(\alpha)})
-((\id_{\thetaalgebra}-\phi)\circ\inverse{h}_*)(\intlambda|_{g}) \\
=& \intlambda|_{g\Phi(\alpha)h\Phi(\beta)}
-(\id_{\thetaalgebra}-\phi)\intlambda|_{gh} \\
=& c|_{(gh,{(\inverse{h}}\groupaction\alpha)\beta)} \\
=& \text{l.h.s.\ of Eq.~\eqref{Eqt:c}}
.\end{align*}
Thus, $\gthetaC_{\intlambda}$ is indeed multiplicative with respect to the group structure.
\newline
\textbf{2)} Let $\partial:\crossedmodulealgebra\to\wedge^{l}(\crossedmodulealgebra)$
be the $l$-differential induced by the multiplicative $l$-vector field $\gthetaC_{\intlambda}$.
According to Lemma~\ref{Lem:partialSigmacalculate}, we have
\begin{equation*}
\partial(x+u)=-\left.\frac{d}{dt}\right|_{t=0} c|_{\exp t(x+u)}, \quad\forall x+u\in\crossedmodulealgebra
.\end{equation*}
Assume that $(\domega_\lambda,\ddelta_\lambda)$ is the infinitesimal data
corresponding to $\gthetaC_{\intlambda}$.
According to Proposition~\ref{prop:groupdifferentialexplicit}, we have
\begin{equation*}
\domega_\lambda (u)= \partial (u) = -\left.\frac{d}{dt}\right|_{t=0}\big(\intlambda|_{ \Phi(\exp tu)}
-(\id_{\thetaalgebra}-\phi)\intlambda|_{\unit_{\ggroup}}\big) = \big(\dlambda\circ\phi\big) (u).
\end{equation*}
Moreover,
\begin{align*}
\ddelta_\lambda (x) &= \pr_{\galgebra\wedge(\wedge^{l-1}\thetaalgebra)}\partial (x) \\
&= -\left.\frac{d}{dt}\right|_{t=0} \pr_{\galgebra\wedge(\wedge^{l-1}\thetaalgebra)}
\big(\intlambda|_{\exp tx}-(\id_{\thetaalgebra}-\phi)\intlambda|_{\exp tx}\big) \\
&= -\left.\frac{d}{dt}\right|_{t=0}\phipush(\intlambda|_{\exp tx}) \\
&= \big(\phipush\circ\dlambda\big)(x)
.\end{align*}
Hence it follows that $\ddelta_\lambda = \phipush\circ\dlambda$.
\newline
\textbf{3)} We first prove the following formula:
\begin{equation}
\label{Eqt:PARTIALVintlamda}
\intsigma|_g = \domega (\intlambda|_g)+\minuspower{(k-1)(l-1)+1}\dlambda(\wedgedelta_g)
-\pr_{{\wedge^{k+l-1}\thetaalgebra }}\ba{\wedgedelta|_g}{\intlambda|_g}
,\end{equation}
where $\wedgedelta: G\to\galgebra\wedge(\wedge^{k-1}\theta)$
is the Lie group $1$-cocycle corresponding to $\delta$ as in Eq.~\eqref{eq:delta2}.
To prove it, assume that
\[ \intlambda|_g=\sum_{i}f_i(g)u_i \quad\text{and}\quad \wedgedelta|_g=\sum_{j}h_j(g)x_j\wedge w_j,
\quad\forall g\in\ggroup ,\]
where $f_i, h_j\in\cinf{\ggroup}$, $u_i\in\wedge^l\thetaalgebra$, $x_j\in \galgebra$, and $w_j\in\wedge^{k-1}\thetaalgebra$.
According to Eq.~\eqref{Eqt:groupoiddifferentiononf} and Proposition~\ref{prop:domegadefined}, we have
\begin{align*}
\intsigma|_g &= \groupoiddifferential{}(\intlambda)|_g \\
&= \sum_{i}\Big(f_i(g)\groupoiddifferential{}(u_i)+(\groupoiddifferential{}f_i)|_g\wedge u_i\Big) \\
&= \sum_{i}\Big(f_i(g)\domega(u_i)+\minuspower{k-1}\inserts_{(L^*_{ {g}}\mathrm{d}f_i)}\wedgedelta|_g\wedge u_i\Big) \\
&= \domega(\intlambda|_g)+\minuspower{k-1}\sum_{i,j}\big(\ip{L^*_g df_i}{x_j}h_j (g)w_j\wedge u_i\big) \\
&= \domega(\intlambda|_g)+\minuspower{k-1}\sum_{j}
\left(h_j(g) w_j\wedge\left.\frac{d}{dt}\right|_{t=0}\intlambda|_{g\exp tx_j}\right) \\
&= \domega(\intlambda|_g)+\minuspower{k-1}\sum_{j}\Big(h_j(g)w_j\wedge\big(-\dlambda(x_j)
-\pr_{{\wedge^{k+l-1}\thetaalgebra }}\ba{x_j}{\intlambda|_{g}}\big)\Big) \\
&= \text{r.h.s.\ of Eq.~\eqref{Eqt:PARTIALVintlamda}}
.\end{align*}
Here in the second from the last equality, we used the identity
\[ \intlambda|_{g\exp tx_j}=\intlambda|_{\exp tx_j}+(\exp tx_j)^{-1}_*\intlambda|_{g} .\]
From Eq.~\eqref{Eqt:PARTIALVintlamda}, it follows that $\intsigma$ is indeed a Lie group $1$-cocycle.
Moreover, the induced Lie algebra $1$-cocycle is
\begin{align*}
\dsigma(x) &= -\left.\frac{d}{dt}\right|_{t=0} \intsigma|_{\exp tx} \\
&= \left.\frac{d}{dt}\right|_{t=0}
\Big(\minuspower{(k-1)(l-1)}\dlambda(\wedgedelta|_{\exp tx})
-\domega(\intlambda|_{\exp tx})
+\pr_{{\wedge^{k+l-1}\thetaalgebra}}\ba{\wedgedelta|_{\exp tx}}{\intlambda|_{\exp tx}}\Big) \\
&= \domega(\dlambda(x))-\minuspower{(k-1)(l-1)}\dlambda(\ddelta(x))
.\end{align*}
Here $-\left.\frac{d}{dt}\right|_{t=0}\pr_{{\wedge^{k+l-1}\thetaalgebra}}
\left[ \wedgedelta|_{\exp tx} , \intlambda|_{\exp tx} \right]=0$, since both
$\wedgedelta$ and $\intlambda$ are group 1-cocycles.
This completes the proof.
\end{proof}

\subsection{Proof of the main theorem}\label{Section:proofofmain}

The following result describes the infinitesimal data of a
quasi-Poisson structure on the 2-group $\crossedmoduletwogroup$.

\begin{proposition}\label{Thm:quasiPoissontotriple}
Let $(\crossedmoduletwogroup,\gthetaPi,\inteta)$ be a quasi-Poisson
2-group as in Definition~\ref{defn:qusi-Poisson}. Let $(\domega,\skewdelta)$
be the corresponding infinitesimal of $\gthetaPi$ and
$\deta:\galgebra\to\wedge^3\thetaalgebra$ the Lie algebra $1$-cocycle induced by $\inteta$.
Then the following identities hold:
\begin{align}\label{Eqn:omega2}
\domega^2&=\deta\circ\phi, \\
\label{Eqn:omegaanddelta}
(\domega+\skewdelta)\circ\skewdelta&=\phipush\circ\deta, \\
\domega\circ\deta&=\deta\circ\ddelta,
\end{align}
where $\deta$ is identified with its extension to a degree-$2$ derivation of the exterior
algebra $\wedge^\bullet(\crossedmodulealgebra)$.
\end{proposition}

\begin{proof}
Let $\gthetaC_{\inteta}=\gpoidleftmove{\inteta}-\gpoidrightmove{\inteta}$.
According to Proposition~\ref{Lem:multiplicativeoutofcocycle},
$\gthetaC_{\inteta}$ is multiplicative. Moreover, its corresponding
infinitesimal is $(\deta\circ\phi,\phipush\circ\deta)$.
By Proposition~\ref{Prop:bracketofmultiplicative}, the infinitesimal
of $\thalf\ba{\gthetaPi}{\gthetaPi}$ is given by
$(\domega^2,(\domega+\skewdelta)\circ\skewdelta)$.
Thus Eq.~\eqref{Eqn:halfPiPiCinteta} implies Eqs.~\eqref{Eqn:omega2}
and~\eqref{Eqn:omegaanddelta}. On the other hand,
Eq.~\eqref{Eqn:Pietazero} is equivalent to
$\groupoiddifferential{\gthetaPi}(\inteta)=0$.
By Lemma~\ref{Lem:multiplicativeoutofcocycle}~(\ref{m3}),
we have $\domega\circ\deta-\deta\circ\ddelta=0$.
This completes the proof.
\end{proof}

Conversely, we have
\begin{proposition}\label{Thm:tripledatatoquasipoisson}
Let $\crossedmoduletwogroup$ be a Lie 2-group.
If both Lie groups $\ggroup$ and $\thetagroup$
are connected and simply connected, every triple $(\domega,\skewdelta,\deta)$,
where $(\domega,\skewdelta)$ satisfies the conditions of
Theorem~\ref{Thm:Sigmatopair}, and $\deta:\galgebra\to\wedge^3\thetaalgebra$
is a Lie algebra $1$-cocycle satisfying the conditions of
Proposition~\ref{Thm:quasiPoissontotriple}, can be uniquely integrated to a
quasi-Poisson structure on $\crossedmoduletwogroup$.
\end{proposition}

\begin{proof}
By Theorem~\ref{Thm:pairtoSigma}, we obtain a multiplicative bivector field
$\gthetaPi$ on $\crossedmoduletwogroup$ whose infinitesimal is $(\domega,\skewdelta)$.
Let $\inteta:\ggroup\to\wedge^3\thetaalgebra$ be the Lie group $1$-cocycle integrating $\deta$.
By Proposition~\ref{Lem:multiplicativeoutofcocycle}~(\ref{m3}),
$\groupoiddifferential{\gthetaPi}(\inteta)$ vanishes since
$\domega\circ\deta-\deta\circ\ddelta=0$.
Thus $\ba{\gthetaPi}{\gpoidleftmove{\inteta}}=\gpoidleftmove{\groupoiddifferential{\gthetaPi}(\inteta)}=0$.
Moreover, Eqs.~\eqref{Eqn:omega2} and~\eqref{Eqn:omegaanddelta} imply
Eq.~\eqref{Eqn:halfPiPiCinteta} according to
Proposition~\ref{Lem:multiplicativeoutofcocycle}~(\ref{m1} and \ref{m2}).
\end{proof}

Finally, we need the following
\begin{lemma}\label{Prop:lie2codatatriple}
A quasi-Lie 2 bialgebra structure on a crossed module of Lie algebras
$\crossedmoduletriple{\thetaalgebra}{\phi}{\galgebra}$ is equivalent to
triples $(\skewdelta,\domega,\deta)$ of linear maps
$\skewdelta:\galgebra\to W_2\subset\galgebra\wedge\thetaalgebra$,
$\domega:\thetaalgebra\to\wedge^2\thetaalgebra$ and
$\deta:\galgebra\to\wedge^3\thetaalgebra$ that satisfy the following properties:
\begin{enumerate}
\item $\phipush\circ\omega=\skewdelta\circ\phi$;
\item $\domega^2=\deta\circ\phi$;
\item $(\domega+\skewdelta)\circ\skewdelta=\phipush\circ\deta$;
\item $\domega\circ\deta=\deta\circ\ddelta$;
\item $\deta$ is a Lie algebra $1$-cocycle;
\item $\skewdelta$ is a Lie algebra $1$-cocycle;
\item $x\moduleaction\omega(u)-\omega(x\moduleaction u)
=\pr_{\wedge^k\thetaalgebra}(\ba{u}{\skewdelta(x)})$, for all $x\in\galgebra$ and $u\in\thetaalgebra$.
\end{enumerate}
\end{lemma}

\begin{proof}
By Proposition~\ref{Prop:Lie2coalgebraelements},
a weak Lie 2-coalgebra structure underlying $\crossedmoduletriple{\thetaalgebra}{\phi}{\galgebra}$
is equivalent to an element
$c=\derivescovarphi+\derivescobracket+ \derivescoaction+\derivescoh\in
\degreesubspace{\symmetricalgebra}{-4}$ such that $\Poissonbracket{c,c}=0$.
Here $\phi$ and $\derivescovarphi$ are related by the equation:
$\phimap(u)=\{{\derivesvarphi},u\}$, for all $u\in\thetaalgebra$.
And $\crossedmoduletriple{\thetaalgebra}{\phi}{\galgebra}$ is a
quasi-Lie 2-bialgebra if and only if
$\Poissonbracket{o+c,o+c}=0$, where $o=\derivesbracket+\derivesaction$ is the data defining
the crossed module structure of $\crossedmoduletriple{\thetaalgebra}{\phi}{\galgebra} $,
as a special Lie 2-algebra with $\derivesh=0$.
Introduce the operators $\skewdelta$, $\domega$, and $\deta$ by the following relations:
\begin{align*}
\pairing{\skewdelta(x)}{\xi\wedge\kappa} &= -\Poissonbracket{\Poissonbracket{\Poissonbracket{\derivescoaction,x},\xi},\kappa}; \\ \pairing{\domega(u)}{\kappa_1\wedge \kappa_2} &=
\Poissonbracket{\Poissonbracket{\Poissonbracket{\derivescobracket,u}\kappa_1},\kappa_2}; \\
\pairing{\deta(x)}{\kappa_1\wedge\kappa_2\wedge\kappa_3} &=
\Poissonbracket{\Poissonbracket{\Poissonbracket{\Poissonbracket{\derivescoh,x},\kappa_1},\kappa_2},\kappa_3}.
\end{align*}
for all $x\in\galgebra$, $u\in\thetaalgebra$, $\xi\in\galgebrastar$, and $\kappa,\kappa_i\in\thetaalgebrastar$.
Expand $\Poissonbracket{o+c,o+c}$ and consider the result term by term.
Immediately, we have the following:
\begin{enumerate}
\item\label{n1} the $(\symmetricproduct^2\galgebra)\symmetricproduct \galgebrastar$-part is zero if and only if $\skewdelta$ is  valued in $W_2$;
\item\label{n2} the $\thetaalgebra \symmetricproduct
\galgebra \symmetricproduct \thetaalgebrastar$-part is zero if and only if Condition~\ref{n1}) is satisfied;
\item\label{n3} the $(\symmetricproduct^3 \thetaalgebra)\symmetricproduct
\thetaalgebrastar$-part is zero if and only if Condition~\ref{n2}) is satisfied;
\item\label{n4} the $(\symmetricproduct^2 \thetaalgebra)\symmetricproduct
\galgebra \symmetricproduct \galgebrastar$-part is zero if and only if Condition~\ref{n3}) is satisfied;
\item\label{n5} the $(\symmetricproduct^4 \thetaalgebra)\symmetricproduct
\galgebrastar$-part is zero if and only if Condition~\ref{n4}) is satisfied;
\item\label{n6} the $(\symmetricproduct^2 \galgebrastar)\symmetricproduct
(\symmetricproduct^3\thetaalgebra)$-part is zero if and only if Condition~\ref{n5}) is satisfied;
\item\label{n7} the $(\symmetricproduct^2\galgebrastar)\symmetricproduct\galgebra\symmetricproduct\thetaalgebra$-part
is zero if and only if Condition~\ref{n6}) is satisfied;
\item\label{n8} the $(\symmetricproduct^2\thetaalgebra)\symmetricproduct
\galgebrastar\symmetricproduct\thetaalgebrastar$-part is zero if and only if Condition~\ref{n7}) is satisfied.
\end{enumerate}
This concludes the proof.
\end{proof}

\begin{proof}[Proof of Theorem \ref{Thm:quasi-lie2-bioneteonequasipoisson}]
Lemma~\ref{Prop:lie2codatatriple} implies that a quasi-Lie 2 bialgebra
underlying the crossed module $\crossedmoduletriple{\thetaalgebra}{\phi}{\galgebra}$
is determined by the triple $(\domega,\skewdelta,\deta)$ that satisfies
the conditions in Theorem~\ref{Thm:Sigmatopair} and Proposition~\ref{Thm:quasiPoissontotriple}.
Thus Theorem~\ref{Thm:quasi-lie2-bioneteonequasipoisson} follows from
Proposition~\ref{Thm:quasiPoissontotriple} and Proposition~\ref{Thm:tripledatatoquasipoisson}.
\end{proof}

\subsection{Coboundary quasi-Poisson structures}

The following proposition describes a class of interesting examples of
quasi-Poisson structures on a Lie 2-group.

\begin{proposition}\label{Prop:etatriple}
\begin{enumerate}
\item Associated to any Lie group $1$-cocycle $\intlambda:\ggroup\to\wedge^2\thetaalgebra$,
there exists a quasi-Poisson structure $\big(\gthetaPi,\inteta\big)$ on $\crossedmoduletwogroup$ given as follows:
\begin{gather}\label{gthetaPiintlambda}
\gthetaPi = \gpoidleftmove{\intlambda}-\gpoidrightmove{\intlambda }, \\
\inteta = \thalf\ba{\intlambda}{\intlambda} . \label{eq:47}
\end{gather}
In Eq.~\eqref{gthetaPiintlambda}, $\intlambda$ is considered as
a section in $\sections{\wedge^2\LiealgebroidoverG}$
and the bracket in Eq.~\eqref{eq:47} stands for the pointwise
Schouten bracket on $\wedge^\bullet\theta$.
\item The infinitesimal
$(\domega_\dlambda,\ddelta_\dlambda,\deta_\dlambda)$ of $\big(\gthetaPi,\inteta\big)$
as described by Proposition~\ref{Thm:quasiPoissontotriple} is as follows:
\begin{gather*}
\domega_{\dlambda}=\dlambda\circ\phi, \\
\skewdelta_{\dlambda}=\phipush\circ\dlambda,\\
\deta_{\dlambda}=\dlambda\circ\phipush\circ\dlambda,
\end{gather*}
where $\dlambda:\galgebra\to\wedge^2\thetaalgebra$ is the Lie algebra $1$-cocycle induced by $\intlambda$.
\end{enumerate}
\end{proposition}

\begin{proof}
The proof is standard, and is left to the reader.
\end{proof}

In particular, any $\Rmatrix\in\wedge^2\thetaalgebra$ induces
a Lie algebra $1$-cocycle $\lambda_\Rmatrix:\galgebra\to\wedge^2\theta$:
\begin{equation}\label{Eqt:dlambda}
\dlambda_{\Rmatrix}(x)=-x\moduleaction{{\Rmatrix}},
\quad\forall x\in\galgebra
.\end{equation}
Therefore by Proposition~\ref{Prop:etatriple}, there exists a quasi-Poisson structure
$(\gthetaPi_\Rmatrix,\inteta_{\Rmatrix})$ on the 2-group $\crossedmoduletwogroup$.
By a straightforward computation, we can describe this quasi-Poisson structure more explicitly:
\begin{multline}\label{Eqt:PioutofRmatrix}
(\gthetaPi_\Rmatrix)|_{(g,\alpha)}
=\fullaction{R_{g*}\Phistar}\Rmatrix-\fullaction{L_{g*}\Phistar}\Rmatrix
+\fullaction{L_{\alpha*} \inverse{g}_*}\Rmatrix-\fullaction{R_{\alpha*} \inverse{g}_*}\Rmatrix \\
+[(L_{g*}\circ\Phistar)\otimes(L_{\alpha})]\Rmatrix-[(R_{g*}\circ\Phistar)\otimes (L_{\alpha}\circ\inverse{g}_*)]\Rmatrix
,\end{multline}
and
\begin{equation}\label{Eqt:intlambda}
(\inteta_\Rmatrix)|_{g}=\thalf\big(\ba{\Rmatrix}{\Rmatrix}-\fullaction{\inverse{g}_*}\ba{\Rmatrix}{\Rmatrix}\big)
.\end{equation}
The infinitesimal of $(\gthetaPi_\Rmatrix,\inteta_{\Rmatrix})$ is given as follows:
\begin{align}
\label{Eqn:omegaRmatrix}
\domega_{\Rmatrix}(u) &= \ba{\Rmatrix}{u}, &&\forall u\in\thetaalgebra; \\
\label{Eqn:deltaRmatrix}
\skewdelta_{\Rmatrix}(x) &= -\phipush(x\moduleaction\Rmatrix)=-x\moduleaction(\phipush\Rmatrix),
&&\forall x\in\galgebra; \\
\deta_{\Rmatrix}(x)  &= - \thalf x\moduleaction{\ba{\Rmatrix}{\Rmatrix}}, &&\forall x\in\galgebra .
\end{align}

According to Lemma~\ref{Prop:lie2codatatriple}, the triple $(\domega_{\Rmatrix},\skewdelta_{\Rmatrix},\deta_{\Rmatrix})$
also defines a quasi-Lie 2-bialgebra structure underlying $\crossedmoduletriple{\thetaalgebra}{ }{\galgebra}$.
In particular, if $\deta_{\Rmatrix}=0$, i.e.\
\begin{equation}\label{Eqt:rrginvariant}
x\moduleaction\ba{\Rmatrix}{\Rmatrix}=0, \quad\forall x\in\galgebra
,\end{equation}
we obtain a Lie 2-bialgebra.

\begin{definition}\label{Defn:Rmatrix}
An element $\Rmatrix$ of $\wedge^2\thetaalgebra$ is called an r-matrix of a Lie algebra crossed module
$\crossedmoduletriple{\thetaalgebra}{\phi}{\galgebra}$ if
$\ba{\Rmatrix}{\Rmatrix}\in\wedge^3\thetaalgebra$ is $\galgebra$-invariant,
i.e.\ if Eq.~\eqref{Eqt:rrginvariant} holds.
\end{definition}

Similar to the Poisson group case, we have the following
\begin{theorem}
Corresponding to any $r$-matrix $\Rmatrix$ as above, there is
\begin{enumerate}
\item a Poisson Lie 2-group structure $\gthetaPi_\Rmatrix$
on $\crossedmoduletwogroup$ such that
\begin{equation*}
\gthetaPi_\Rmatrix=\gpoidleftmove{\intlambda_\Rmatrix}-\gpoidrightmove{\intlambda_\Rmatrix},
\end{equation*}
where $\intlambda_\Rmatrix:\ggroup\to\wedge^2\thetaalgebra$
is given by Eq.~\eqref{Eqt:dlambda} and
\item a Lie bialgebra crossed module underlying
$\crossedmoduletriple{\thetaalgebra}{\phi}{\galgebra}$.
\end{enumerate}
\end{theorem}

In this case, the Lie bracket on $\thetaalgebrastar$ is induced by the $r$-matrix $\Rmatrix$:
\[ \pairing{\ba{\kappa_1}{\kappa_2}_{\Rmatrix}}{u}
=\pairing{\kappa_1\wedge\kappa_2}{[\Rmatrix,u]},
\quad\forall\kappa_1,\kappa_2\in\thetaalgebrastar,u\in\thetaalgebra ,\]
while the action of $\thetaalgebrastar$ on $\galgebrastar$ is given by
\[ \pairing{\kappa\moduleaction\xi}{x}=\pairing{{\kappa}\wedge{\phi^*\xi}}{x\moduleaction\Rmatrix} ,\]
for all $\kappa\in\thetaalgebrastar$, $\xi\in\galgebrastar$, and $x\in\galgebra$.

\begin{example}
Let $\thetaalgebra=\gltwo\cong\reals\id{}\oplus\sltwo$ and $\galgebra=\sltwo$.
Then the projection $\phi:\gltwo\to\sltwo$ is a Lie algebra crossed module.
It is easy to check that any $\Rmatrix\in\wedge^2\thetaalgebra$ is indeed an r-matrix.
\end{example}

\begin{example}
Let $\galgebra$ be a Lie algebra and $\thetaalgebra\subseteq\galgebra$ an ideal.
Consider the Lie algebra crossed module $\iota:\thetaalgebra\to\galgebra$, where $\iota$ is the inclusion.
Assume that $\Rmatrix\in\wedge^2\thetaalgebra$ such that
$\ba{\Rmatrix}{\Rmatrix}\in\wedge^3\thetaalgebra$ is $\galgebra$-invariant.
Then $\Rmatrix$ is clearly an r-matrix.
For example, we take $\galgebra=\gltwo$ and $\thetaalgebra=\sltwo$.
Then any bivector in $\wedge^2 \thetaalgebra$ is indeed an r-matrix.
\end{example}

\nocite{*}

\bibliographystyle{amsplain}
\bibliography{references}
\end{document}